\documentclass[9pt,reqno]{amsart}  
\usepackage[english]{babel}
\usepackage[p,osf]{cochineal}
\usepackage[scale=.95,type1]{cabin}
\usepackage[zerostyle=c,scaled=.94]{newtxtt}
\usepackage[T1]{fontenc} 
\usepackage[utf8]{inputenc} 
\usepackage{amsthm}

\usepackage{amsmath}
\usepackage{amssymb}
\usepackage{pgfplots} 
\pgfplotsset{compat=newest} 
\usepgfplotslibrary{fillbetween,colorbrewer}
\usepackage{amsfonts}
\usepackage{amsaddr}
\usepackage{xspace,subcaption}
\usepackage{enumitem} 
\usepackage{csquotes}
\usepackage{xcolor}
\usepackage[colorlinks]{hyperref}
\hypersetup{colorlinks, breaklinks, citecolor=teal,filecolor=black}
\usepackage[nosort,capitalise]{cleveref}
\crefname{enumi}{}{}
\crefname{equation}{}{}
\usepackage{graphicx}
\usepackage{mathtools}
\mathtoolsset{centercolon}
\usepackage{bm}
\usepackage{pgfkeys}
\usepackage{pgf,interval}
\intervalconfig{soft open fences}

\DeclareMathOperator{\Tr}{Tr}

\DeclareMathOperator{\sgn}{sgn}

\DeclareMathOperator{\NCP}{NC}
\DeclareMathOperator{\NCG}{NCG}

\DeclareMathOperator{\E}{\mathbf{E}}
\DeclareMathOperator{\Prob}{\mathbf{P}}

\DeclareMathOperator{\pTr}{pTr}

\newcommand{\ov}{\overline}
\newcommand{\ii}{\mathrm{i}}

\newcommand{\F}{\mathbf{F}}

\ifdefined\C
\renewcommand{\C}{\mathbf{C}}
\else
\newcommand{\C}{\mathbf{C}}
\fi

\newcommand{\un}{\underline}
\newcommand{\vx}{\bm{x}}

\newcommand{\vl}{{\bm{l}}}

\newcommand{\vj}{{\bm{j}}}

\newcommand{\vy}{\bm{y}}

\newcommand{\wt}{\widetilde}

\newcommand{\cE}{\mathcal{E}}
\newcommand{\av}{\mathrm{av}}
\newcommand{\iso}{\mathrm{iso}}
\newcommand{\R}{\mathbf{R}}
\newcommand{\N}{\mathbf{N}}
\newcommand{\Z}{\mathbf{Z}}

\newcommand{\cG}{\mathcal{G}}
\newcommand{\cO}{\mathcal{O}}
\newcommand{\co}{{\scriptstyle\mathcal{O}}}

\newcommand{\dif}{\operatorname{d}\!{}}
\DeclarePairedDelimiter{\braket}{\langle}{\rangle}%
\DeclarePairedDelimiter{\abs}{\lvert}{\rvert}%
\DeclarePairedDelimiter{\norm}{\lVert}{\rVert}%
\providecommand\given{}
\newcommand\SetSymbol[1][]{\nonscript\:#1\vert\allowbreak\nonscript\:\mathopen{}}
\DeclarePairedDelimiterX{\tuple}[1](){\renewcommand\given{\SetSymbol[\delimsize]}#1}
\DeclarePairedDelimiterX{\set}[1]\{\}{\renewcommand\given{\SetSymbol[\delimsize]}#1}
\DeclarePairedDelimiterXPP{\landauO}[1]{\cO}(){}{#1}
\DeclarePairedDelimiterXPP{\landauo}[1]{\co}(){}{#1}
\DeclarePairedDelimiterXPP{\landauOprec}[1]{\cO_\prec}(){}{#1}

\usepackage[giveninits=true,url=false,doi=false,isbn=false,eprint=true,datamodel=mrnumber,sorting=nty,maxcitenames=4,maxbibnames=99,backref=false,block=space,backend=biber,bibstyle=phys]{biblatex} 
\AtEveryBibitem{\clearfield{month}}
\AtEveryCitekey{\clearfield{month}} 
\renewbibmacro{in:}{}
\DeclareFieldFormat[article]{title}{\emph{#1}} 
\DeclareFieldFormat{mrnumber}{\ifhyperref{\href{http://www.ams.org/mathscinet-getitem?mr=#1}{\nolinkurl{MR#1}}}{\nolinkurl{#1}}}
\DeclareFieldFormat{pmid}{\ifhyperref{\href{https://www.ncbi.nlm.nih.gov/pubmed/#1}{\nolinkurl{PMID#1}}}{\nolinkurl{#1}}}
\DeclareFieldFormat{eprint}{\ifhyperref{\href{https://arxiv.org/abs/#1}{\nolinkurl{arXiv:#1}}}{\nolinkurl{#1}}}
\renewbibmacro*{doi+eprint+url}{%
  \iftoggle{bbx:doi}{\printfield{doi}}{}
  \newunit\newblock%
  \printfield{mrnumber}%
  \newunit\newblock%
  \printfield{pmid}%
  \newunit\newblock%
  \printfield{eprint}%
  \iftoggle{bbx:url}{\usebibmacro{url+urldate}}{}}

\bibliography{bibclt}

\numberwithin{equation}{section} 
\newtheorem{theorem}{Theorem}[section]
 
\newtheorem{lemma}[theorem]{Lemma}
\newtheorem{proposition}[theorem]{Proposition}
\newtheorem{remark}[theorem]{Remark} 
\newtheorem{corollary}[theorem]{Corollary}
\newtheorem{definition}[theorem]{Definition}

\date{\today}
\author{Giorgio Cipolloni}
\address{Princeton Center for Theoretical Science, Princeton University, Princeton, NJ 08544, USA}
\author{L\'aszl\'o Erd\H{o}s\(^\#\)}
\address{IST Austria, Am Campus 1, 3400 Klosterneuburg, Austria}
\author{Dominik Schr\"oder\(^{\ast}\)}
\address{Institute for Theoretical Studies, ETH Zurich, Clausiusstr.\ 47, 8092 Zurich, Switzerland}
\email{gc4233@princeton.edu} 
\email{lerdos@ist.ac.at}
\email{dschroeder@ethz.ch}
\thanks{\(^\#\)Supported by ERC Advanced Grant ``RMTBeyond'' No.~101020331}
\thanks{\(^\ast\)Supported by Dr.\ Max R\"ossler, the Walter Haefner Foundation and the ETH Z\"urich Foundation}
\subjclass[2010]{60B20, 15B52} 
\keywords{Global Law, Local Law, Random Matrices}
\title{Optimal multi-resolvent local laws for Wigner matrices}
\date{\today} 

\begin{document} 
\thispagestyle{empty}

\begin{abstract} 
    We prove local laws, i.e.\ optimal concentration estimates for   arbitrary products of resolvents 
    of a Wigner random matrix with  deterministic matrices in between. 
    We find  that the size of such products heavily depends on 
    whether some of the deterministic matrices are traceless.
    Our estimates correctly  account for this dependence
    and they hold optimally 
    down to the smallest possible spectral scale.
\end{abstract}
\maketitle

\section{Introduction}
A remarkable feature of large  Hermitian random matrices $H$ is that their
resolvents $G(z)=(H-z)^{-1}$ tend to concentrate around a deterministic matrix $M=M(z)$ 
for spectral parameters $z\in\C$ even just slightly away from the real axis.
If the correlation among the matrix entries of $H$ is sufficiently weak,
in particular for \emph{Wigner matrices} with independent (up to Hermitian symmetry)
and identically distributed matrix elements, this phenomenon  
holds as long as $\abs{\Im z}$ is just slightly above the typical eigenvalue spacing around $\Re z$.
While the random matrix $H$ strongly fluctuates around its mean $\E H$, 
it is surprising that the resolvent has such a strong concentration 
property even on small spectral scales. 
Rigorous results of this type are generally called \emph{local laws} and
they play a fundamental role in random matrix theory since they
are able to resolve spectral properties of $H$ almost down to individual eigenvalues.
We remark that for Wigner matrices $M(z)=m(z) I$  is 
the multiple of the identity matrix, where $m$ is the Stieltjes transform of
Wigner's semicircle distribution.  For more general ensembles $M$ is given as the solution 
of the  \emph{(matrix) Dyson equation}, a non-linear deterministic equation~\cite{MR3916109}.

Historically, the primary motivation for local laws was to provide the necessary a priori 
estimates  in the \emph{three step strategy} to prove the Wigner-Dyson-Mehta
spectral universality for random matrices via the Dyson Brownian Motion  (DBM), see~\cite{MR3699468} for a comprehensive summary. The first local law was  proved  for Wigner matrices in the tracial sense~\cite{MR2481753}; extended later to more general \emph{entry-wise}~\cite{MR2981427}  and \emph{isotropic}~\cite{MR3103909} senses, as well as to  much more general  classes of random matrices,
including nonzero expectation~\cite{MR3134604, MR3502606,MR3800833}, nontrivial variance profile~\cite{MR3719056}, and even correlations~\cite{MR3916109,MR3941370}. Numerous related works  focused on local laws for band matrices~\cite{MR3085669,MR3934695,MR4221663, 2104.12048, 2107.05795}, sparse matrices~\cite{MR3098073,MR3688032, MR3800840,MR3962004,MR4135670,MR3800840},
heavy tails~\cite{MR3129806,MR3916110}, accurate error terms~\cite{MR3405612,MR4125959},
general
invariant $\beta$-ensembles~\cite{MR4389077, MR2905803, MR3192527, MR3253704, MR4168391, MR4278668, MR4009708, MR3755113, MR2966359}
and many more.  

With a very few recent exceptions, listed at the end of Section~\ref{sec:rule},
all local laws so far concerned a single resolvent. Their {\it averaged} and {\it isotropic}
versions  assert that
for any fixed $\epsilon>0$, deterministic test matrix $B$ and test vectors $\vx, \vy$, the bounds
\begin{equation}\label{1G}
    \abs{\braket{(G(z)-M(z))B}} \le \frac{N^\epsilon \norm{B}}{N\eta}, \qquad 
    \abs{\braket{\vx, (G(z)-M(z))\vy}} \le \frac{N^\epsilon\norm{\vx}\norm{\vy}}{\sqrt{N\eta}},\qquad  \eta:=\abs{\Im z}
\end{equation}
hold with very high probability, where $N$ is a dimension of $H$, $\braket{R}:= \frac{1}{N}\Tr R$
denotes the normalized trace and $\braket{\cdot ,\cdot }$ denotes the scalar product in $\C^N$.
The estimates~\eqref{1G} are
optimal in the critical small $\eta$ regime (up to the factor $N^\epsilon$).

This paper is concerned with the multi-resolvent generalizations of~\eqref{1G}.
If $G$  is approximated by  $M$, 
what approximates the square of the resolvent? 
The naive answer  $G^2\approx M^2$ is wrong, even for the simplest Wigner case
since the approximation $G\approx M$ in~\eqref{1G} holds true only in weak sense; it cannot be
``squared''.  Nevertheless $G^2$ still concentrates
and the hint given by the identity 
$G(z)^2=\partial_z G(z)$ leads to the correct answer. Indeed $G(z)^2\approx \partial_z M(z)$ in the sense
\begin{equation}\label{2G}
    \abs{\braket{(G(z)^2-\partial_z M(z))B}} \le \frac{N^\epsilon\norm{B}}{N\eta^2}, \qquad  
    \abs{\braket{\vx, (G(z)^2-\partial_z M(z))\vy}} \le \frac{N^\epsilon\norm{\vx}\norm{\vy}}{\sqrt{N}\eta^{3/2}},
\end{equation}
and again the error terms are optimal.
Note that these  error terms match the differentiation procedure; indeed~\eqref{2G} can formally be 
obtained by ``differentiating''~\eqref{1G}.

Such algebraic ideas, however, do not help much further if we ask for  concentration of the alternating  product 
\begin{equation}\label{GBG}
    G(z_1)B_1 G(z_2) B_2 G(z_3)\ldots B_{k-1} G(z_k) 
\end{equation} 
of resolvents and deterministic matrices $B_1, B_2, \ldots $, 
and more generally for
\begin{equation}\label{fBf}
    f_1(H) B_1 f_2(H) B_2 \ldots B_{k-1} f_k(H),
\end{equation} 
where $f_i$'s are  arbitrary functions on $\R$.
The product~\eqref{GBG} still concentrates but its
deterministic  approximation, denoted by  $M(z_1, B_1, z_2, \ldots, B_{k-1}, z_k)$,
is non-trivial even for the Wigner case and it was identified only recently in~\cite[Theorem 3.4]{MR4372147}
(however, formulas for traces of~\eqref{fBf} when $f_i$'s are polynomials have already  been obtained 
within free probability theory, see e.g.~\cite[Theorem 5.4.5]{MR2760897} or \cite[Sect 4. Thm 20.]{MR3585560}).
The main result of the current work is to prove the optimal error term for this approximation
and thus to establish the optimal local law for any product of the type~\eqref{GBG} when $H$ is from the  Wigner ensemble
(Theorem~\ref{thm local law}).
These optimal multi-resolvent local laws will then be used to establish the universality of the Gaussian fluctuations  of~\eqref{fBf}
in subsequent works. To keep the current paper focused, we present here  only one simple application
of our new local law to
improve  our control on the  thermalisation effect of the Wigner matrices (see Remark~\ref{rmk:therm} below).

In connection with CLT for linear eigenvalue statistics, special cases of tracial  local laws for~\eqref{GBG} for $k=2, 3$ 
have been proven 
in~\cite{2103.05402,MR3678478,MR3805203, MR4095015,MR4119592,MR4255183,MR4255226}.  
These results, however, 
considered the special $B_i=I$ case, where resolvent identities can directly reduce the number of $G$'s.
More importantly, the accurate analysis of the case with general $B$'s must handle traceless $B$'s separately
as we explain in the next subsection.

\subsection{The role of the traceless matrices}\label{sec:rule}
The major complication for the multi-resolvent local law is that the size of  $M(z_1, B_1, z_2, \ldots, B_{k-1}, z_k)$
heavily depends on whether some of the matrices $B_i$  are traceless or not,
and the  error term must match the size of $M$ to be considered optimal.
For example, if  $B_1=B_2=\ldots  = B_{k-1}=I$, then
$\braket{M(z_1, I, z_2, \ldots, I, z_k)}\sim (1/\eta)^{k-1}$   with $\eta:= \min |\Im z_i|$  in the interesting
regime where $\eta\lesssim 1$,  
and the corresponding local law
\begin{equation}\label{GGGG}
    |\braket{G(z_1) G(z_2)G(z_3)\ldots G(z_k)-M(z_1, I, z_2, \ldots, I, z_k)} | \le \frac{N^\epsilon}{N\eta^k} = \frac{1}{\eta^{k-1}} \frac{N^\epsilon}{N\eta}
\end{equation}
is optimal (up  to $N^\epsilon$) for $\eta\lesssim1$. 
 Note that the error term is by a factor $N^\epsilon/N\eta$ smaller than the  deterministic approximation,
hence~\eqref{GGGG} proves concentration for any $\eta\gg 1/N$.

Exactly the same estimate  holds for \eqref{GBG}
with general deterministic matrices
$B_i$  with $\| B_i\|=1$ instead of $B_i=I$, see~\cite[Theorem 3.4]{MR4372147}.
However, if all $B_1, B_2, \ldots, B_k$ are traceless, $\braket{B_i}=0$, then  
in the  $\eta\lesssim 1$  regime typically 
\begin{equation}\label{Mz}
    \braket{M(z_1, B_1, z_2, \ldots, B_{k-1}, z_k)B_k}\sim \frac{1}{\eta^{\lfloor k/2\rfloor-1}},
\end{equation}
therefore  $N^\epsilon/(N\eta^k)$  in~\eqref{GGGG} is much
bigger than the deterministic approximation. %
This indicates that the  robust error term 
proven in~\cite[Theorem 3.4]{MR4372147} for general  matrices is far from being optimal
when traceless matrices are involved, but it does not give a hint what the optimal error term should be.

The correct answer, in a heuristic form, can be formulated by 
the following rule of thumb that we coin  the $\sqrt{\eta}$-rule (in the $\eta\lesssim 1$ regime):
\begin{itemize}
    \item[\textbf{\(\sqrt{\bm\eta}\)-rule:}] Each traceless matrix $B_i$ reduces both the size of $M$ and the error term 
    by a factor $\sqrt{\eta}$.
\end{itemize}
Establishing the $\sqrt{\eta}$-rule for $M$ is relatively  straightforward
given its explicit form, but for the error term it is much harder -- this is the main content
of the current paper. 

The  special role of a traceless deterministic matrix even for the single resolvent local law
was observed only  recently in~\cite{2012.13218}, where  it was shown that  
\[
\abs{\braket{(G(z)-m(z))B}} =  \abs{\braket{G(z)B} } \le \frac{N^\epsilon}{N\sqrt{\eta}}
\]
if $\braket{B}=0$ in contrast to the much bigger error of order $1/(N\eta)$ for general $B$ in~\eqref{1G}. In fact, $G-m$ has two different fluctuation modes, a tracial and a traceless one,
expressed somewhat informally in the following two-scale central limit theorem  
\begin{equation}\label{2CLT}
    \braket{(G(z)-m(z))B} %
    \approx \braket{B}\frac{\xi_1}{N\eta} + \braket{ \mathring{B}\mathring{B}^*}^{1/2} \frac{\xi_2}{N\sqrt{\eta}}
\end{equation}
where $\xi_1$ and $\xi_2$ are independent Gaussian variables and
$\mathring{B}:= B-\braket{B}$ is the traceless part of $B$. The asymptotics $\approx$ in~\eqref{2CLT}
is understood in the sense of all moments and in the limit as $N\eta\gg1$; see \cite[Theorem 4.1]{MR4372147} 
for the precise statement. 

Tracking the influence of the traceless deterministic matrices in multi-resolvent local laws for Wigner matrices
played an essential role in our proof of the {\it Eigenstate thermalisation hypothesis}~\cite{MR4334253}, and
in the {\it functional central limit theorems} to understand the fluctuation modes of 
$f(W)$ as a matrix \cite{2012.13218}. 
However, in these papers only two-  and three-resolvent local laws were necessary and suboptimal error was sufficient.
For example, a key technical ingredient  in~\cite{MR4334253} was the local law
\begin{equation}\label{GBGB}
    \braket{G(z)BG^*(z)B} = |m(z)|^2 \braket{BB^\ast} + O\Big( \frac{N^\epsilon}{\sqrt{N\eta}}\Big)
\end{equation}
for any $\braket{B}=0$ with $\|B\|\lesssim 1$, which in particular implied the upper bound
$$  
\braket{G(z)BG(z)B} = \landauO{1}, \qquad \mbox{for $N\eta\ge N^{2\epsilon}$}
$$
in agreement with~\eqref{Mz} applied to $k=2$.  In the relevant small $\eta$ regime 
the error in~\eqref{GBGB} is better  than the 
robust  error of order $1/(N\eta^2)$ from~\eqref{GGGG} valid irrespective whether $B$ is traceless or not,
but~\eqref{GBGB} is still far from  optimal. The $\sqrt{\eta}$-rule   predicts  an error term of order $1/(N\eta)$
in~\eqref{GBGB},  a factor of $(\sqrt{\eta})^2$ better than the robust error~\eqref{GGGG}, 
while~\eqref{GBGB} does not even get the optimal $N$-power that is naturally expected in the $\eta\sim 1$
regime.
Similarly, specific three-resolvent local laws that were proven in~\cite[Proposition 3.4]{2012.13218},
also came with suboptimal errors.
Finally, we mention a related two-resolvent local law for the  Hermitization
of an i.i.d. matrix in~\cite[Theorem 5.2]{1912.04100} where the mechanism 
for the reduced error term is different from the $\sqrt{\eta}$-rule.

\subsection{Strategy of the proof} We developed a  very concise  new method to 
prove multi-resolvent local laws. 
The basic idea for all local law proofs is to show that
$G$, or in the multi-resolvent case $GBGB\ldots G$ from~\eqref{GBG}, approximately satisfies the Dyson equation, 
the defining equation of the corresponding $M$. 
In the previous approaches the fluctuating error term in this  approximation was treated separately and
it was shown to be negligible with the help of a high moment cumulant expansion.
The expansion generated many terms and a fairly involved Feynman diagrammatic representation  
was needed to bookkeep and estimate them. This becomes especially cumbersome where
some additional smallness effect needs to be consistently tracked along the whole expansion.
For example, in the main technical Theorem 4.1 in~\cite{MR4334253}, we  meticulously counted
the number of ``effectively'' traceless $B$ factors, struggling with the complication that some $B$ factor becomes $B^2$ along the cumulant expansion, losing its smallness effect. Even suboptimal 
error terms for small $k$ as in~\eqref{GBGB} required major efforts and the general case was out of reach.

Our new method drastically simplifies this procedure using two unrelated ideas. First, 
the large Feynman diagrammatic representation is actually due to an overexpansion
of the fluctuating error term which  can be considerably reduced if one expands ``minimalistically'', so to say.
In the context of single resolvent  averaged local laws this idea appeared first in~\cite{MR3800840},
coined as \emph{recursive moment estimates}, we will use 
this philosophy for the multi-resolvent situation and also for the isotropic
case. 

Second, the fundamental concern in  the proofs of multi-resolvent local laws is how to truncate the resulting hierarchy
involving longer and longer chains of the form $GBGB\ldots G$. The cumulant expansion
for a chain of length $k$   as in~\eqref{GBG} will contain chains of length up to $2k$.
For the single resolvent  local law, $k=1$, this problem is usually solved by the 
Ward identity $GG^*= \Im G/\eta$, immediately  reducing longer chains to a single resolvent.
If  traceless matrices are in between $G$'s such identity is not directly applicable.
In~\cite{MR4334253} we solved this problem by  considering the positive quantity
$\Lambda^2:=\braket{\Im G B \Im G B}$ for traceless $B$ and estimated all longer chains
in terms of $\Lambda$, to  arrive, finally,  at a simple Gromwall-type  inequality for $\Lambda$, roughly of the type
\begin{equation}\label{Grom}
    \Lambda^2\lesssim 1 + \frac{\Lambda^2}{N\eta},
\end{equation}
from which $\Lambda\lesssim 1$ immediately follows.
The reduction of longer chains to $\Lambda$'s  involved a careful  Schwarz inequality within
the spectral decomposition  of $H$, for example for an averaged chain involving $2k$ resolvents
(using  $\Im G$'s instead of $G$ for illustrational simplicity) we used
\begin{equation}\label{4G}
    \begin{split}
        |\braket{ (\Im G\,  B)^{2k}}| 
        &= \frac{1}{N} \Big| \sum_{i_1\ldots i_{2k}} \langle {\bm u}_{i_1}, B{\bm u}_{i_2} \rangle  \langle {\bm u}_{i_2}, B{\bm u}_{i_3} 
        \rangle \ldots
        \langle {\bm u}_{i_{2k}}, B{\bm u}_{i_1}\rangle %
        \prod_{j=1}^{2k} \Im \frac{1}{\lambda_{i_j}-z} %
        \Big| \\
        & \le  \frac{1}{N} \Big( \sum_{ij}  |\langle {\bm u}_i, B{\bm u}_j\rangle |^2  \Im \frac{1}{\lambda_i-z}\Im \frac{1}{\lambda_j-z}\Big)^k\\
        &= N^{k-1} \braket{\Im G B \Im G B}^k.
    \end{split}
\end{equation}
Here $\lambda_i$ and  ${\bm u}_i$ 
are the eigenvalues and the orthonormal eigenvectors of $H$, respectively. The size of the l.h.s., based upon
its deterministic approximation~\eqref{Mz}, is $\eta^{-k+1}$, while the r.h.s. is of order $N^{k-1}$
hence this inequality lost a factor $(N\eta)^{k-1}$. 
Very roughly, each summation in~\eqref{4G} effectively runs over $N\eta$ different $i$ indices
and if each summand were independent, then an effective central limit theorem would reduce
the size by a factor $1/\sqrt{(N\eta)^{2k}}=(N\eta)^{-k}$,  in reality this effect is  weaker
by a factor $N\eta$. Nevertheless, for larger $k$'s this loss  in the Schwarz inequality
in~\eqref{4G} cannot be recovered from the smallness of higher order cumulants,
which eventually results in suboptimal error terms
in the local law in~\cite{MR4334253}. Another complication is that the bound~\eqref{4G} is 
also needed for $(GB)^k$. Since spectrally  $G$  is much less localized than $\Im G$, 
technically we could not do the analysis  locally in the spectrum and $\Lambda$ was actually 
defined after taking a supremum over the real parts of the  spectral parameters $z_i$ 
in $G$'s.

The basic objects in the current paper are %
the appropriately  rescaled versions of the {\it differences} $(GB)^k-M_kB$
between  alternating chains  of length $k$ and
their deterministic counterparts $M_k$. %
More precisely, we set
\begin{equation}\label{psi}
    \Psi^\av_k: = N\eta^{k/2}| \braket{(GB)^k-M_kB}|,
\end{equation}
and its isotropic version $\Psi^\iso_k$ is defined similarly. The general definition allows for
different spectral parameters and different $B$ matrices in the $GBGBGB....$ chain
but we ignore this technicality here. 
The rescaling is chosen such that $\Psi_k^{\av, \iso}\lesssim 1$ 
corresponds to the optimal local laws to be proven.

The "minimalistic"  cumulant expansion applied directly  to the moments of  $\Psi$'s
generates further chains of alternating products of resolvents and $B$'s.
Each of them is expressed as their deterministic "main term" $M$ plus
the error term involving  $\Psi$'s, i.e. for this purpose  we write~\eqref{psi} as 
$$
\braket{(GB)^k} = \braket{M_kB} + \mathcal{O} \Big(\frac{ \Psi^\av_k}{N\eta^{k/2}}\Big),
$$
and similarly for matrix elements $[(GB)^k]_{ab}$. 
The explicit $M_k$ terms can be directly 
estimated, leaving us with  a nonlinear infinite hierarchy
of coupled {\it master inequalities} for $\Psi^\av_k$ and $\Psi^\iso_k$ for each $k$ (Proposition~\ref{prop Psi ineq}).
The estimate for $\Psi_k$ still contains terms involving $\Psi_{2k}$ 
since the cumulant expansion generates longer chains. 
This time, however, we truncate the hierarchy
in the most economical way; roughly speaking a chain 
of length $2k$ is split into two chains of length $k$ instead
of $k$ chains of length two as in~\eqref{4G}.  Hence
many fewer $N\eta$ factors are lost in the analogue of~\eqref{4G}; the 
loss is only $(N\eta)^2$ for the averaged bounds and $N\eta$ in the isotropic bound,
independently of $k$ (see Lemma~\ref{lemma doubling} below).

Even after the reduction of longer chains to shorter ones, the
new truncated system of master inequalities cannot be closed by a simple algebra,
in contrast to the  single inequality~\eqref{Grom} derived for $\Lambda$. %
We first prove a non-optimal {\it a priori} bound $\Psi_k^{\av, \iso}\lesssim \sqrt{N\eta}$ {\it for all $k$} with 
a step-two induction argument and successively improving the power of $N\eta$ in each step.
Then we start the procedure all over again, but now we will not use the reduction of $\Psi_{2k}$'s back to 
$\Psi_k$'s that would cost us  $(N\eta)$ or $(N\eta)^2$ factors; we rather use the 
already  proven a priori bound $\Psi_{2k}\lesssim \sqrt{N\eta}$
that loses only $\sqrt{N\eta}$. It turns out that such a loss can finally be compensated by the 
smaller size of the higher cumulants.

Summarizing, the key conceptual novelty in the current approach compared with~\cite{MR4334253} is twofold.
First, in~\cite{MR4334253} we operated with upper bounds on size of the chains, like~\eqref{4G}, 
while now we operate on the level of the much more precise $\Psi$'s measuring the fluctuations
of the chains, i.e. their deviations from their deterministic counterpart. 
This enables us to determine the leading order term for resolvent chains of any length, and perform a more accurate analysis purely on the level of sub-leading deviations.
Second, longer chains
are split only into two smaller chains, yielding much less $(N\eta)$-factors lost.
However, the price  for this higher accuracy is that we need to 
handle a new infinite system of inequalities for the $\Psi$'s. Finally, two important
technical differences are that (i) we can work locally in the spectrum and (ii) now we use the minimalistic cumulant expansion 
that considerably shortens the argument.

\subsection*{Notation and conventions}
We introduce some notations we use throughout the paper. For integers \(l,k\in\N \) we use the notations \([k]:= \set{1,\ldots, k}\), and
\[ 
[k,l):=\{ k, k+1,\ldots, l-1\}, \qquad [k,l]:=\{ k, k+1,\ldots, l-1, l\}
\]
for \(k< l\). By $\lceil \cdot \rceil$, $\lfloor \cdot \rfloor$ we denote the upper and lower integer part, respectively, i.e. for $x\in \mathbf{R}$ we define $\lceil x\rceil:=\min\{ m\in\mathbf{N}: \, m\ge x\}$ and $\lfloor x\rfloor:=\max\{ m\in\mathbf{N}: \, m\le x\}$. For positive quantities \(f,g\) we write \(f\lesssim g\) and \(f\sim g\) if \(f \le C g\) or \(c g\le f\le Cg\), respectively, for some constants \(c,C>0\) which depend only on the constants appearing in the moment condition, see~\eqref{mom} later.
 We denote vectors by bold-faced lower case Roman letters \({\bm x}, {\bm y}\in\C ^N\), for some \(N\in\N\). Vector and matrix norms, \(\norm{\vx}\) and \(\norm{A}\), indicate the usual Euclidean norm and the corresponding induced matrix norm. For any \(N\times N\) matrix \(A\) we use the notation \(\braket{ A}:= N^{-1}\Tr  A\) to denote the normalized trace of \(A\). Moreover, for vectors \({\bm x}, {\bm y}\in\C^N\) and matrices  \(A\in\C^{N\times N}\) we define
\[ \braket{ {\bm x},{\bm y}}:= \sum_{i=1}^N \overline{x}_i y_i, \qquad A_{\vx\vy}:=\braket{\vx,A\vy}.\]

We will use the concept of ``with very high probability'' meaning that for any fixed \(D>0\) the probability of an \(N\)-dependent event is bigger than \(1-N^{-D}\) if \(N\ge N_0(D)\). Moreover, we use the convention that \(\xi>0\) denotes an arbitrary small constant which is independent of \(N\). We introduce the notion of \emph{stochastic domination} (see e.g.~\cite{MR3068390}): given two families of non-negative random variables
\[
X=\tuple*{ X^{(N)}(u) \given N\in\N, u\in U^{(N)} }\quad\text{and}\quad Y=\tuple*{ Y^{(N)}(u) \given N\in\N, u\in U^{(N)} }
\] 
indexed by \(N\) (and possibly some parameter \(u\)) we say that \(X\) is stochastically dominated by \(Y\), if for all \(\epsilon, D>0\) we have \[
\sup_{u\in U^{(N)}} \Prob\left[X^{(N)}(u)>N^\epsilon  Y^{(N)}(u)\right]\leq N^{-D}
\]
for large enough \(N\geq N_0(\epsilon,D)\). In this case we use the notation \(X\prec Y\) or \(X= \landauOprec*{Y}\).

\section{Main results}\label{sec:norm}
We start with the definition of the matrix model we consider.  
\begin{definition}
    \label{def:wigner}
    We call $W$ a Wigner matrix
    if it is an $N\times N$ random Hermitian matrix which satisfies the following properties. The off-diagonal matrix elements below the diagonal are centred independent, identically distributed (i.i.d) real (\(\beta=1\)) or complex \((\beta=2)\) random variables with 
    $\E\abs{w_{ij}}^2=1/N$. Additionally, in the complex case we assume that \(\E w_{ij}^2=0\). The diagonal elements are centred i.i.d.\ real 
    random variables with \(\E  w_{ii}^2 = 2/(N\beta)\). Furthermore, we assume that for every  $q\in N$  there is a constant $C_q$ such that
    \begin{equation}
        \label{mom}
        \E \abs{ \sqrt{N} w_{ij}}^{q}\le C_q.
    \end{equation}
\end{definition}
\begin{remark}
    The assumptions \(\E w_{ij}^2=0\) in the complex case, and \(\E w_{ii}^2=2/(\beta N)\) are made to make the presentation clearer. All our results can be easily extended to this case as well, but we refrain from doing it for notational simplicity.
\end{remark}
We set $G(z):= (W-z)^{-1}$ to be resolvent of the Wigner matrix $W$ with spectral parameter $z\in\C\setminus\R$. 
The \emph{optimal local law} asserts that \(G(z)\) is approximately equal to \(m(z)I\) down to the microscopic scale \(\abs{\Im z}\gg 1/N\), where 
\begin{equation}
    m(z)= m_\mathrm{sc}(z):=\int_{-2}^2\frac{1}{x-z}\rho_\mathrm{sc}(x)\dif x,\qquad \rho_\mathrm{sc}(x):=\frac{\sqrt{4-x^2}}{2\pi}
\end{equation} 
is the Stieltjes transform of the semicircular distribution. 
\begin{theorem}\label{thm sG local law}
    For any \(z\in\C\setminus\R\) with $|z|\le N^{100}$, 
     \(d:=\mathrm{dist}(z,[-2,2])\), \(\eta:=\abs{\Im z}\) and any deterministic vectors \(\vx,\vy\) it holds that 
    \begin{equation}
        \abs{\braket{G-m}}\prec  \begin{cases}
            \frac{1}{N\eta}, & d<1\\ 
            \frac{1}{Nd^2}, & d\ge 1,
        \end{cases} \quad \abs{\braket{\vx,(G-m)\vy}}\prec  \norm{\vx}\norm{\vy}\begin{cases}
            \frac{\sqrt{\abs{\Im m(z)}}}{\sqrt{N\eta}} +\frac{1}{N\eta}, & d<1\\ 
            \frac{1}{\sqrt{N}d^{2}}, & d\ge 1.
        \end{cases}
    \end{equation}
\end{theorem} 
\cref{thm sG local law} in this form, including both the $d<1$ and
$d\ge 1$ regimes, can be found in~\cite[Theorem 2.1]{MR3941370} even for much more general
random matrix ensembles allowing for correlations. Its tracial version and its special entry-wise version
(where $\vx, \vy$ are coordinate vectors) have already been established in~\cite[Lemma B.1]{MR3949949}.
However, the really interesting $d<1$ regime has been proven much earlier: tracial version in~\cite{MR2481753},
entry-wise version in~\cite{MR2981427}  and isotropic version~\cite{MR3103909}; with many other refinements
and generalisations mentioned in the introduction. The $d\ge 1$ regime, sometimes called the {\it global law},
 is much easier and most papers on the local law  naturally  excluded it for convenience albeit they could have
 handled this regime, too, with some minor extra effort. 
 
In case of several spectral parameters $z_1, z_2,\ldots$ we use the abbreviation $G_i:= G(z_i)$. 
For our main result we recall from~\cite{MR4372147} that the deterministic approximation to \(G_1B_1 G_2 \cdots G_{k-1}B_{k-1}G_k\) for arbitrary deterministic matrices \(B_1,\ldots,B_{k-1}\) is given by 
\begin{equation}\label{eq M def}
    M(z_1,B_1,\ldots,B_{k-1},z_k) := \sum_{\pi\in\NCP[k]}\mathrm{pTr}_{K(\pi)}(B_1,\ldots,B_{k-1}) \prod_{B\in\pi} m_\circ[B],
\end{equation}
where \(\NCP[k]\) denotes the non-crossing partitions of the set \([k]=\set{1,\ldots,k}\) arranged in increasing order, and \(K(\pi)\) denotes the Kreweras complement of \(\pi\) \cite{MR309747}, e.g.\ \(K(\set{134|2|5|6})=\set{12|3|456}\). Moreover, the partial trace \(\mathrm{pTr}_{\pi}\) with respect to a partition \(\pi\) is given by 
\begin{equation}\label{eq pTr def}
    \mathrm{pTr}_\pi(B_1,\ldots,B_{k-1}) = \prod_{B\in\pi\setminus B(k)}\braket*{\prod_{j\in B}B_j}\prod_{j\in B(k)\setminus\set{k}} B_j,
\end{equation} 
with \(B(k)\in\pi\) denoting the unique block containing \( k\). Finally for any subset \(B\subset [k]\) we define \(m[B]:=m_\mathrm{sc}[z_B]\) as the iterated divided difference of \(m_\mathrm{sc}\) evaluated in \(z_B:=\set{z_i\given i\in B}\), and by \(m_\circ[\cdot ]\) denote the free-cumulant transform of \(m[\cdot]\) which is defined implicitly by the relation 
\begin{equation}
    \label{eq:freecumulant}
    m[B] = \sum_{\pi\in\NCP(B)} \prod_{B'\in\pi} m_\circ[B'],  \qquad \forall B\subset [k],
\end{equation}
e.g.\ \(m_\circ[i,j]=m[\set{i,j}]-m[\set{i}]m[\set{j}]\). We note that the iterated divided difference admits the representation 
\begin{equation}\label{msc dd}
    m_\mathrm{sc}[\set{z_i\given i \in B}] = \int_{-2}^2\rho_\mathrm{sc}(x)\prod_{i\in B}\frac{1}{(x-z_i)}\dif x.
\end{equation}
For more details on these notations, see~\cite[Section 2]{MR4372147}.
As an example we have 
\begin{equation}
    \begin{split}
        M(z_1,B_1,z_2) &= \braket{B_1}(m_\mathrm{sc}[z_1,z_2]-m_\mathrm{sc}(z_1)m_\mathrm{sc}(z_2)) + B_1 m_\mathrm{sc}(z_1)m_\mathrm{sc}(z_2) \\
        &= \frac{\braket{B_1}}{2\pi}\int_{-2}^2\frac{\sqrt{4-x^2}}{(x-z_1)(x-z_2)}\dif x + (B_1-\braket{B_1}) m_\mathrm{sc}(z_1)m_\mathrm{sc}(z_2)
    \end{split}
\end{equation}
for any matrix \(B_1\) and 
\begin{equation}
    \begin{split}
        & M(z_1,A_1,z_2,A_2,z_3) \\
        &\quad = \braket{A_1 A_2} (m_\mathrm{sc}[z_1,z_3]-m_\mathrm{sc}(z_1)m_\mathrm{sc}(z_3))m_\mathrm{sc}(z_2)+ A_1 A_2 m_\mathrm{sc}(z_1)m_\mathrm{sc}(z_2)m_\mathrm{sc}(z_3)\\
        &M(z_1,A_1,z_2,A_2,z_3,A_3,z_4) \\
        &\quad= \braket{A_1A_2A_3}(m_\mathrm{sc}[z_1,z_4]-m_\mathrm{sc}(z_1)m_\mathrm{sc}(z_4))m_\mathrm{sc}(z_2)m_\mathrm{sc}(z_3) \\
        &\qquad + A_1A_2A_3 m_\mathrm{sc}(z_1)m_\mathrm{sc}(z_2)m_\mathrm{sc}(z_3)m_\mathrm{sc}(z_4)\\
        &\qquad +A_1\braket{A_2A_3}(m_\mathrm{sc}[z_2,z_4]-m_\mathrm{sc}(z_2)m_\mathrm{sc}(z_4))m_\mathrm{sc}(z_1)m_\mathrm{sc}(z_3)\\
        &\qquad+A_3\braket{A_1A_2}(m_\mathrm{sc}[z_1,z_3]-m_\mathrm{sc}(z_1)m_\mathrm{sc}(z_3))m_\mathrm{sc}(z_2)m_\mathrm{sc}(z_4)
    \end{split}
\end{equation}
for traceless matrices \(A_1,A_2,A_3\). 
In the sequel we  follow the notational convention that general deterministic matrices are denoted by $B$, while
the letter $A$ is used to denote
explicitly traceless  matrices.

We now give bounds on the size of the deterministic term $M(z_1,B_1\ldots,z_{k},B_k,z_k)$. The proof of this lemma is presented in Appendix~\ref{sec:addproofs}.

\begin{lemma}\label{lemma size M}
    If \(a\) out of the \(k\) matrices \(B_1,\ldots,B_k\) with \(\norm{B_i}\lesssim1\) are traceless,
    i.e. \(\braket{B_j}=0\) holds for $a$ different indices (for some \(0\le a\le k\)), then it holds that 
    \begin{equation}\label{Msize}
        \begin{split}
            \norm{M(z_1,B_1\ldots,z_{k},B_k,z_{k+1})}&\lesssim  \begin{cases} \frac{1}{\eta^{k-\lceil a/2\rceil}}  &  d\le 1 \\ 
              \frac{1}{d^{k+1}}  &  d\ge 1 ,
              \end{cases} 
            \\ 
            \abs{\braket{M(z_1,B_1,\ldots,z_{k-1},B_{k-1},z_k)B_k}}&\lesssim 
              \begin{cases} \frac{1}{\eta^{k-1-\lceil a/2\rceil}}  &  d\le 1 \\ 
              \frac{1}{d^{k}}  &  d\ge 1 ,
              \end{cases} 
        \end{split}
    \end{equation}
    with $\eta:= \min_j |\Im z_j|$ and $d:= \min_j \mathrm{dist}(z_j, [-2,2])$. 
   Generically, both bounds are  sharp 
    when not all \(\Im z_i\) have the same sign. 
\end{lemma}

\begin{theorem}[Multi-resolvent local law]\label{thm local law} 
    Fix \(\epsilon>0\), let \(k\ge 1\) and consider \(z_1,\ldots,z_{k+1}\in\C\) with
    $\max_j |z_j|\le N^{100}$, \(  \min_j\abs{\Im z_j}\ge N^{-1+\epsilon}\), 
    and let \(B_1,\ldots,B_k\) be deterministic matrices of norm \(\norm{B_j}\lesssim 1\), such that $a$  of 
    them are traceless for some \(0\le a\le k\).  Let $\eta:= \min_j\abs{\Im z_j}$ and  $d:= \min_j \mbox{dist}(z_j, [-2,2])$. 
    Then for arbitrary deterministic vectors \(\vx,\vy\) of norm \(\norm{\vx}+\norm{\vy}\lesssim 1\) 
     we have the   \emph{optimal averaged local law}
       \begin{subequations}\label{eq local laws}
        \begin{equation}\label{loc1}
            \abs{\braket{G_1B_1\cdots G_k B_k-M(z_1,B_1,\ldots,B_{k-1},z_k)B_k}} \prec
             \begin{cases} \frac{1}{N\eta^{k-a/2}}  &  d\le 1 \\ 
              \frac{1}{Nd^{k+1}}  &  d\ge 1 ,
              \end{cases}
        \end{equation}
        and the \emph{optimal isotropic local law}
        \begin{equation}
            \label{loc2}
            \abs*{\braket*{\vx,\Bigl(G_1B_1 \cdots G_k B_k G_{k+1}-M(z_1,B_1,\ldots,B_{k},z_{k+1})\Bigr)\vy}} \prec
             \begin{cases}  \frac{1}{\sqrt{N}\eta^{k-a/2+1/2}} &  d\le 1 \\ 
              \frac{1}{\sqrt{N} d^{k+2}}  &  d\ge 1,
              \end{cases}
        \end{equation}
    \end{subequations}
    where $G_j:=G(z_j)$. %
\end{theorem}
\begin{remark}~
    \begin{enumerate}[label=(\alph*)]
        \item In the regime \(d\le 1\)  the error terms in~\cref{eq local laws} are generically smaller,
        by a factor of \(1/(N\eta)\) and \(1/\sqrt{N\eta}\), respectively,  than the leading terms
         \(\braket{\vx,M_{k+1}\vy}\) and \(\braket{M_kB_k}\), with the shorthand notation
          \(M_j:=M(z_1,B_1,\ldots,B_{j-1},z_j)\), c.f.~\cref{lemma size M}. 
          For \(d\ge 1\) the error terms are smaller by a factor \(1/(Nd)\) and \(1/(\sqrt{N}d)\), respectively. 
        \item The estimates~\cref{loc1,loc2} are optimal. This can  be easily  seen from the proof since in the Gaussian 
        case the leading term of the variance~\cref{av Gaussian bound,wt Xi iso} is estimated 
        sharply due to the optimality of~\cref{lemma size M}. 
        \item The  really interesting part of        \cref{thm local law} is the
       \(d\le 1\) regime, 
         since the effect of traceless matrices is only relevant when at least some of the spectral parameters is
         close to the limiting spectrum $[-2,2]$. 
         In fact,  for \(d\ge 1\) very similar bounds were already given in~\cite[Theorem 3.4]{MR4372147}.
         However, the proof in~\cite{MR4372147}  relied on the fairly involved 
         diagrammatic expansion used in~\cite[Theorem 4.1]{MR4334253}. With our new method, we can  give a much shorter 
         alternative proof for this regime as well; this will be explained separately in Appendix~\ref{app:largeeta}.
        \item  With our new method we could also present a simplified proof of the single resolvent local law as
        stated in~\cref{thm sG local law}.    In this way we could circumvent citing the quite involved~\cite[Theorem 2.1]{MR3941370}
           that was designed to handle much more general ensembles than Wigner. 
           The proof of the easier $d\ge1$  regime is especially simple in this new way, which would eliminate
           the main reason for citing~\cite{MR3941370} instead of earlier and simpler single resolvent local law proofs for $d\le 1$. 
           For the sake of brevity we refrain from reproving~\cref{thm sG local law},  
        and instead  we assume it as an input within the proof of~\cref{thm local law}.   
    \end{enumerate}
\end{remark}

By Theorem~\ref{thm local law} we will also conclude the following corollary.
\begin{corollary}\label{cor:HS}
    Let $k\ge 3$, let $B_1,\dots,B_k$ be deterministic matrices with $\norm{B_i}\lesssim 1$, such that $a$ of them
    are traceless for some $0\le a \le k$.
     Let $f_1,\dots,f_k$ be Sobolev functions $f_i\in H^{\lceil k-a/2\rceil}(\mathbf{R})$ such that 
    $\norm{f_i}_{L^\infty}\lesssim 1$. 
    Then for any deterministic vectors ${\bm x}, {\bm y}$ with $\norm{\bm x}+\norm{\bm y}\lesssim 1$ we have
    \begin{align}\nonumber
        \braket{f_1(W)B_1 \ldots f_k(W)B_k } &= \sum_{\pi\in\NCP[k]}\braket{B_1,\ldots,B_k}_{K(\pi)}\prod_{B\in \pi} \mathrm{sc}_\circ[B]  
        +\mathcal{O}_\prec\left(\frac{\max_i\norm{f_i}_{H^{\lceil k-a/2\rceil}}}{N}\right)\\\label{eq:niceintform}
        \braket{\vx,f_1(W)B_1 \ldots f_k(W)\vy } &= \sum_{\pi\in\NCP[k]}\braket{\vx,\mathrm{pTr}_{K(\pi)}(B_1,\ldots, B_{k-1})\vy} 
        \prod_{B\in \pi} \mathrm{sc}_\circ[B]  \\\nonumber
        &\qquad + \mathcal{O}_\prec\left(\frac{\max_i\norm{f_i}_{H^{\lceil k-a/2\rceil}}}{N^{1/2}}\right),
    \end{align}
    where \(\mathrm{sc}_\circ\) is the free cumulant function from \eqref{eq:freecumulant} of \(\mathrm{sc}[i_1,\ldots,i_n]:=\braket{f_{i_1}f_{i_2}\cdots f_{i_n}}_\mathrm{sc}\), with $\braket{f}_{\mathrm{sc}}:=\int f(x)\rho_{\mathrm{sc}}(x)\dif x$. 
    For \(k=2\) and $a=0,1$ exactly the same result holds. In the remaining case $k=2$, $a=2$ \eqref{eq:niceintform} also
     holds with \(f_i\in H^{\lceil k-a/2\rceil}\) and \(\norm{\cdot}_{H^{\lceil k-a/2\rceil}}\) 
     replaced by \(f_i\in H^2\) and \(\norm{\cdot}_{H^2}\), respectively. 
    The results  in \eqref{eq:niceintform} can be extended straightforwardly
     to include several independent Wigner matrices (see \cite[Remark 2.13]{MR4372147}).
\end{corollary}

Exactly the same result \eqref{eq:niceintform} for $k=1$ and $f\in H^2$ was proven in \cite{2012.13218}, 
where we actually even proved a CLT for $\braket{f(W)A}$.

We remark that in Corollary~\ref{cor:HS} there is a significant improvement in the error term compared
 to \cite[Theorem 2.6]{MR4372147} where the matrices $B_j$ do not necessarily have trace zero. Namely, the Sobolev 
 norm $\norm{\cdot}_{H^k}$ in the error term of \cite[Theorem 2.6]{MR4372147} 
 is here replaced by $\norm{\cdot}_{H^{\lceil k-a/2\rceil}}$, with $a$ denoting the number of traceless matrices. 
 For $a=0$ the error terms in Corollary~\ref{cor:HS} coincide with the ones in \cite[Theorem 2.6]{MR4372147}.

\begin{remark}[Thermalisation]\label{rmk:therm}
    We now specialise Corollary~\ref{cor:HS} to $f(x)=e^{\ii sx}$, with $s>0$, and define
    \begin{equation}
        \varphi(s):=\int_{-2}^2e^{\ii s x}\rho_{\mathrm{sc}}(x)\, \dif x=\frac{J_1(2s)}{s},
    \end{equation}
    where $J_1$ is the Bessel function of the first kind.
    The thermalisation result from~\cite[Corollaries 2.9-2.10]{MR4372147}  asserts that the unitary Heisenberg evolution
    generated by the Wigner matrix renders deterministic observables (matrices) asymptotically independent 
    for large times. More precisely,
    \begin{equation}
        \label{eq:oldtherm}
        \braket{e^{\ii s W}B_1e^{-\ii s W}B_2}=\braket{B_1}\braket{B_2}+\varphi(s)^2\braket{B_1B_2}+\mathcal{O}_\prec\left(\frac{s^2}{N}\right),
    \end{equation}
    for any deterministic matrices $B_1,B_2$ (for simplicity we only stated the case $k=2$).

    Using the optimal local law for two resolvents in \eqref{loc1}, by a very similar proof to the one of Corollary~\ref{cor:HS}, we conclude
    \begin{equation}
        \label{eq:therm}
        \braket{e^{\ii s W}A_1e^{-\ii s W}A_2}=\varphi(s)^2\braket{A_1A_2}+\mathcal{O}_\prec\left(\frac{s}{N}\right),
    \end{equation}
    with $\braket{A_1}=\braket{A_2}=0$. Note the improved error term in \eqref{eq:therm} compared to $s^2 N^{-1}$ from \eqref{eq:oldtherm}, which allow us to prove that
    \[
    \braket{e^{\ii s W}A_1e^{-\ii s W}A_2}\approx \varphi(s)^2\braket{A_1A_2}
    \]
    for any $s\ll N^{1/4}$ (instead of $s\ll N^{1/5}$ from \eqref{eq:oldtherm}), where we used that $\varphi(s)^2\sim s^{-3}$ for $s\gg 1$. We remark that by Corollary~\ref{cor:HS} we obtain a similar improvement for any $k\ge 3$, but we refrain from stating it for notational simplicity.
\end{remark}

\section{Proof of the multi-resolvent local law in the $d\le1$ regime}
We  give a detailed proof  of Theorem~\ref{thm local law}  for the much more involved $d\le1$ regime,
in particular in this case $\eta\le 1$.
In Appendix~\ref{app:largeeta} we explain the necessary modifications for  the $d\ge1$ case.
At a certain technical point (within the proof of Lemma~\ref{absG lemma}), the proof for the $d\le 1$ 
uses~\eqref{loc1} for the $d\ge1$ regime, but this lemma is not needed for the proof in the $d\ge 1$
regime, so our argument is not circular.  With the exception of Appendix~\ref{app:largeeta}, throughout the rest
of the paper we assume that $d\le 1$, hence $\eta\le 1$.

For traceless deterministic matrices \(A_j\), \(\norm{A_j}\le 1\), \(\braket{A_j}=0\), deterministic bounded vectors 
\(\vx,\vy\), \(\norm{\vx}+\norm{\vy}\le1\) and for $k\ge  1$ we introduce the normalized differences
\begin{equation}\label{Psi def}
    \begin{split}
        \Psi_k^\av(\bm z_k,\bm A_k)&: = N \eta^{k/2}\abs*{\braket{G_1A_1\cdots G_kA_k-M(z_1,A_1,\ldots,A_{k-1},z_k)A_k}}, \\
        \Psi_k^\iso(\bm z_{k+1},\bm A_k,\vx,\vy)&: = \sqrt{N\eta^{k+1}}\abs*{\Bigl(G_1A_1\cdots A_k G_{k+1}-M(z_1,A_1,\ldots,A_k,z_{k+1})\Bigr)_{\vx\vy}},
    \end{split}
\end{equation}
where 
\begin{equation}\label{defs}
    G_k:=G(z_k),\quad \eta:=\min_i\abs{\Im z_i},\quad \bm z_k:=(z_1,\ldots,z_k), \quad \bm A_k:=(A_1,\ldots, A_k).
\end{equation}
For convenience we extend these definitions to \(k=0\) by
\begin{equation}
    \Psi_0^\av(z):=N\eta\abs{\braket{G(z)-m_\mathrm{sc}(z)}},\quad \Psi_0^\iso(z,\vx,\vy):=\sqrt{N\eta} \abs{\braket{\vx,(G(z)-m_\mathrm{sc}(z))\vy}}, \quad \eta:=\abs{\Im z},
\end{equation}
and note that 
\begin{equation}\label{eq singleG}
    \Psi_0^\av+\Psi_0^\iso\prec 1
\end{equation} 
by the well known single-resolvent local law~\cite{MR3183577, MR2871147, MR3103909}. 
Note that the index $k$ counts the number of traceless matrices. %

 For notational convenience we also introduce the concept of \emph{\(\epsilon\)-uniform bounds}.
\begin{definition}\label{def uniform} Fix any \(\epsilon>0\) and \(l\in\N\).  Let \(k\in \N\), then  we say that the bounds 
    \begin{equation}\label{general error}
        \begin{split}
            \abs*{\braket*{G(z_1)B_1\cdots G(z_k)B_k-M(z_1,B_1,\ldots,B_{k-1},z_k)B_k}}&\prec \cE^\av,\\
            \abs*{\bigl(G(z_1)B_1\cdots B_kG(z_{k+1})-M(z_1,B_1,\ldots,z_k,B_k,z_{k+1})\bigr)_{\vx\vy}}&\prec \cE^\iso
        \end{split}
    \end{equation}
    hold \emph{\((\epsilon,l)\)-uniformly} for some control parameters \(\cE^{\av/\iso}=\cE^{\av/\iso}(N,\eta) \), 
     depending  only on \(N,\eta\), if the implicit constants in~\eqref{general error} 
     are uniform in bounded deterministic matrices \(\norm{B_j}\le 1\), 
    deterministic vectors \(\norm{\vx}, \norm{\vy}\le1\), and spectral parameters \(z_j\) with \(1\ge \eta:=\min_{j}\abs{\Im z_j}\ge l N^{-1+\epsilon}\),  
    \( \abs{z_j}\le N^{100}\). For simplicity, we say~\cref{general error} holds \(\epsilon\)-uniformly if it holds \((\epsilon,1)\)-uniformly.
    Moreover, we may allow for additional restrictions on the deterministic matrices,  and talk
    about  uniformity under the additional assumption that some of the matrices are traceless, 
    or some of them is a multiple of the identity matrix, etc.
\end{definition}

Note that~\eqref{general error} is stated for each fixed choice of the spectral parameters $z_j$ in the left hand side, but in
fact it is equivalent to an apparently stronger statement, when the same bounds hold with suprema over the spectral parameters \(z_j\).
More precisely, if $ \cE^\av\ge N^{-C}$ for some constant $C$, then~\eqref{general error} implies
\begin{equation}\label{withsup}
    \sup_{z_1, z_2, \ldots, z_k} \abs*{\braket*{G(z_1)B_1\cdots G(z_k)B_k-M(z_1,B_1,\ldots,B_{k-1},z_k)B_k}}\prec \cE^\av
\end{equation}
(and similarly for the isotropic bound),  
where the supremum is taken over all choices of $z_j$'s in the admissible spectral domain, i.e. with  \( \abs{z_j}\le N^{100}\) and 
\(1\ge \min_{j}\abs{\Im z_j}\ge lN^{-1+\epsilon}\). This bound follows from~\eqref{general error} by the usual {\it grid argument}.
Indeed,  we may apply~\eqref{general error} for a dense $N^{-10k}$-grid of $k$-tuples of complex numbers within the 
spectral domain.
The number of such tuples is at most polynomial in $N$ and we use the standard property of
stochastic domination to conclude $\max_i X_i\prec C$ from $X_i\prec C$ as long as the number of $i$'s
is at most polynomial in $N$. Finally, we can use
the Lipschitz  continuity  (with Lipschitz constant at most $\eta^{-k-1}\le N^{k+1}$)
of the left hand side of~\eqref{general error} to extend the bound for  all spectral parameters in the spectral domain.
In the sequel we will frequently use this equivalence between~\eqref{general error} and~\eqref{withsup}, e.g.
when we integrate such bounds over some spectral parameter.

We first establish the following key lemma which allows us 
to conclude multi-resolvent local laws for general deterministic matrices from the special case where 
each deterministic matrix is traceless. 
\begin{lemma}\label{G^2 lemma}
    Fix \(\epsilon>0,l\in\N\) and \(k>0\) and assume that for all \(1\le j \le k\) and some control parameters $\psi_j^{\av/\iso}$ 
    the a priori bounds 
    \begin{equation}\label{eq a priori j1}
        \Psi_j^\av(\bm z_j,\bm A_j)\prec \psi_j^\av,\qquad \Psi_j^\iso(\bm z_j,\bm A_{j},\vx,\vy)\prec \psi_j^\iso
    \end{equation}
    have been established \emph{\((\epsilon,l)\)-uniformly in traceless matrices}. Then it holds that
    \begin{equation}\label{eq G^2 lemma}
        \begin{split}
            \braket{G(z_1)B_1\cdots G(z_k)B_{k}}&=\braket{M(z_1,B_1,\ldots,B_{k-1},z_k)B_k} + \landauOprec*{\frac{\sum_{j=a}^{k-b}\psi_j^\av}{N\eta^{k-a/2}}}\\
            \Bigl(G(z_1)B_1G(z_2)\cdots B_{k}G(z_{k+1})\Bigr)_{\vx\vy} &= M(z_1,B_1,\ldots,B_k,z_{k+1})_{\vx\vy} + \landauOprec*{\frac{\sum_{j=a}^{k-b}\psi_j^\iso}{\sqrt{N}\eta^{k-a/2+1/2}}},
        \end{split}
    \end{equation}
    \((\epsilon,l+1)\)-uniformly in
    vectors $\vx, \vy$  and deterministic matrices \(B_1,\ldots,B_k\), out of which \(0\le a\le k\) are traceless and \(0\le b\le k\) are a multiple of the identity. 
\end{lemma}
Using~\Cref{G^2 lemma} we reduce~\Cref{thm local law} to the following Lemma. 
\begin{lemma}[Final estimate on \(\Psi_k^{\av/\iso}\)]\label{final lemma}
    For any \(\epsilon>0\) and \(k\ge 1\) we have 
    \begin{equation}\label{eq final bound}
        \Psi_k^{\av}+\Psi_k^\iso\prec1
    \end{equation}        
    \(\epsilon\)-uniformly in traceless matrices.
\end{lemma}
\begin{proof}[Proof of~\Cref{thm local law}]
    \Cref{thm local law} is equivalent to~\Cref{final lemma} in case when all matrices are traceless. The general case follows from~\Cref{G^2 lemma} and setting \(\psi_k^{\av/\iso}=1\) due to~\Cref{final lemma}. 
\end{proof}
We prove~\Cref{final lemma} in two steps and first establish a weaker bound as stated in the following lemma.
\begin{lemma}[A priori estimate on \(\Psi_k^{\av/\iso}\)]\label{a priori lemma}
    For any \(\epsilon>0\) and \(k\ge 1\) we have 
    \begin{equation}\label{eq a priori final bound}
        \Psi_k^{\av}+\Psi_k^\iso\prec\sqrt{N\eta}
    \end{equation}        
    \(\epsilon\)-uniformly in traceless matrices.
\end{lemma}
The rest of the proof is organised as follows: First, we prove~\Cref{G^2 lemma}, then in~\Cref{sec master} we state
the \emph{master inequalities} on the \(\Psi_k^{\av/\iso}\) parameters, which we then use to prove~\Cref{final lemma,a priori lemma} in~\Cref{sec lemmata}. Finally, the proof of the master inequalities will be presented in~\Cref{sec master proof}.
\begin{proof}[Proof of~\Cref{G^2 lemma}]
    We start the proof by splitting all those  \(k-a-b\) matrices \(B_i\) that are neither traceless nor multiples of the identity as \(B_i = \braket{B_i} + \mathring{B_i}\). Since~\cref{eq M def}  is multi-linear in the \(B\)-matrices and the error terms in~\cref{eq G^2 lemma} are monotonically decreasing as \(a\) or \(b\) are increased, it is sufficient to prove~\Cref{G^2 lemma} for the special case when \(a+b=k\), i.e.\ all matrices are either traceless or multiple  of the identity.
    
    Moreover, if \(\Im z_i \Im z_j<0\) then we use the resolvent identity \(G(z_i)G(z_j)=[G(z_i)-G(z_j)]/(z_i-z_j)\) and \(\abs{z_i-z_j}\ge \eta\) repeatedly to further reduce the lemma to the special case 
    \begin{equation}\label{lemma reduc1}
        \biggl(\prod_{j=1}^{k_1}G(z_{1,j})\biggr) A_1 \biggl(\prod_{j=1}^{k_2}G(z_{2,j})\biggr) A_2 \cdots
    \end{equation}
    where \(\braket{A_i}=0\) and \(\sgn(\Im z_{i,1})=\cdots=\sgn(\Im z_{i,k_i})\) for all $i$.
     We note that~\cref{eq M def} satisfies the same relation since 
    \begin{equation}
        M(\ldots,z_i,I,z_{i+1},\ldots)= \frac{M(\ldots,z_i,\ldots)-M(\ldots,z_{i+1},\ldots)}{z_i-z_{i+1}}
    \end{equation}
    due to  
    \begin{equation}
        m[z_i,z_{i+1}]=\frac{m_\mathrm{sc}(z_i)-m_\mathrm{sc}(z_{i+1})}{z_i-z_{i+1}}
    \end{equation}
    by definition. Finally, from the residue theorem we have that%
    \begin{equation}\label{lemma reduc2}
        \prod_{j=1}^{k}G(z_{j}) = \frac{1}{\pi}\int_\R \Im G(x+\ii\zeta) \prod_{j=1}^{k}\frac{1}{x-z_j+\sgn(\Im z_j)\ii\zeta}\dif x
    \end{equation} 
    whenever \(0<\zeta<\min_j\Im z_j\) or \(\max_j \Im z_j<-\zeta<0\). We note that \(M\) from~\cref{eq M def} satisfies the same relation since 
    \begin{equation}\label{lemma reduc2M}
        M(\ldots,z_i,I,z_{i+1},I,\ldots,I,z_{i+n},\ldots)= \frac{1}{2\pi\ii}\int_\R\frac{M(\ldots,x+\ii\zeta,\ldots)-M(\ldots,x-\ii\zeta,\ldots)}{(x+\sigma\ii\zeta-z_i)\cdots(x+\sigma\ii\zeta-z_{i+n})}\dif x
    \end{equation}
    for \(\sigma=\sgn(\Im z_i)=\cdots=\sgn(\Im z_{i+n})\) due to multi-linearity and 
    \begin{equation}
        m[z_i,\ldots,z_{i+n}] = \frac{1}{2\pi\ii}\int_\R \frac{m(x+\ii\zeta)-m(x-\ii\zeta)}{(x+\sigma\ii\zeta-z_1)\cdots(x+\sigma\ii\zeta-z_n)}\dif x
    \end{equation}
    from the residue theorem. By using~\cref{lemma reduc2} for each product in~\cref{lemma reduc1} obtain an alternating chain of traceless matrices and resolvents, so that the bound follows by the assumptions in~\eqref{eq a priori j1}.  
\end{proof}
\subsection{Master inequalities and reduction lemma}\label{sec master}
From now on every deterministic matrix \(A_i\) is assumed to be traceless and uniformity is understood as uniformity in traceless matrices. 
\begin{proposition}[A priori estimates on \(\Psi^{\av/\iso}\)]~\label{prop Psi ineq} 
    \begin{enumerate}[label=(\roman*)]
        \item Assume that
        \begin{equation}\label{eq a priori 4}
            \Psi_j^{\av/\iso} \prec \psi_j^{\av/\iso}, \qquad 1\le j\le4
        \end{equation}
        \((\epsilon,l)\)-uniformly. Then it holds that 
        \begin{subequations}
            \begin{align}\label{eq Psi1av}
                \Psi_1^\av&\prec 1+\frac{\psi_{1}^\iso + (\psi_1^\av)^{1/2}+(\psi_2^\av)^{1/2}}{\sqrt{N\eta}} \\\label{eq Psi2av}
                \Psi_2^\av&\prec 1+\psi_1^\av+\frac{\psi_{2}^\iso + (\psi_2^\av)^{1/2}+(\psi_4^\av)^{1/2}}{\sqrt{N\eta}}+\frac{(\psi_1^\iso)^2+(\psi_1^\av)^2+\psi_1^\iso(\psi_2^\av)^{1/2}}{N\eta} \\\label{eq Psi1iso}
                \Psi_1^\iso&\prec 
                1+\frac{\psi_1^\iso+\psi_1^\av}{\sqrt{N\eta}}+\frac{(\psi_2^\iso)^{1/2}}{(N\eta)^{1/4}}\\\label{eq Psi2iso}
                \Psi_2^\iso&\prec 1+\psi_1^\iso+\frac{\psi_2^\iso+(\psi_1^\iso\psi_3^\iso)^{1/2}+\psi_1^\av}{\sqrt{N\eta}}+\frac{\psi_1^\iso\psi_1^\av}{N\eta}+\frac{(\psi_3^\iso)^{1/2}+(\psi_4^\iso)^{1/2}}{(N\eta)^{1/4}},
            \end{align}
        \end{subequations}
        \((\epsilon,l+1)\)-uniformly. 
        \item Now, let \(k>2\) and assume that a priori bounds 
        \begin{equation}\label{eq a priori k}
            \begin{split}
                \Psi_j^\av&\prec \begin{cases}
                    \psi_j^\av:= \sqrt{N\eta}, &j\le k-2,\\ \psi_j^\av,&k-1\le j\le 2k,
                \end{cases}\\
                \Psi_j^\iso&\prec \begin{cases}
                    \psi_j^\iso:= \sqrt{N\eta}, &j\le k-2,\\ \psi_j^\iso,&k-1\le j\le 2k,
                \end{cases}
            \end{split}
        \end{equation}
        have been established \((\epsilon,l)\)-uniformly. Then it holds that 
        \begin{subequations}
            \begin{align}\label{eq Psikav}
                \Psi_k^\av&\prec1+\sum_{j=1}^{k-1}\psi_j^\av+\frac{\psi_{k-1}^\iso+\psi_{k}^\iso + \sum_{j=\lceil k/2\rceil}^{k}(\psi_{2j}^\av)^{1/2}}{\sqrt{N\eta}} \\\label{eq Psikiso}
                \Psi_k^\iso&\prec 1+ \sum_{j=1}^{k-1}\psi_{j}^\iso + \frac{\psi_k^\iso+(\psi_{k+1}^\iso\psi_{k-1}^\iso)^{1/2}+\psi_{k-1}^\av}{\sqrt{N\eta}} +\frac{\sum_{j=k+1}^{2k}(\psi_{j}^\iso)^{1/2}}{(N\eta)^{1/4}  }   
            \end{align} 
        \end{subequations}
      \((\epsilon,l+1)\)-uniformly.
    \end{enumerate}
\end{proposition}
Since in~\Cref{prop Psi ineq} resolvent chains of length \(k\) are estimated by resolvent chains of length up to \(2k\) 
we will need the following {\it reduction lemma}  in order avoid an infinite hierarchy of inequalities with higher and higher $k$-indices.
\begin{lemma}[Reduction inequality]\label{lemma doubling}
    Fix \(k\ge 1\) and assume that \(\Psi_n^{\av/\iso}\prec \psi_n^{\av/\iso}\) holds for \(0\le n\le 2k\) 
    \((\epsilon,l)\)-uniformly. Then it holds that 
    \begin{align}\label{eq doubling}
        \Psi_{2k}^\av &\prec\begin{cases}
            (N\eta)^2 + (\psi_k^\av)^2, &\text{\(k\) even} \\
            (N\eta)^2 + N\eta (\psi_{k-1}^\av+\psi_{k+1}^\av) + \psi_{k-1}^\av\psi_{k+1}^\av, &\text{\(k\) odd}.
        \end{cases}\intertext{\((\epsilon,l)\)-uniformly. Moreover,  for \(j\le k\) and for \(k\) even, we have}\label{eq doubling iso}
        \Psi_{k+j}^\iso&\prec
        N\eta\Big(1+\frac{\psi_k^\iso}{\sqrt{N\eta}}\Big)\Big(1+ \frac{ (\psi_{2j}^\av)^{1/2}}{\sqrt{N\eta}}\Big),
    \end{align}
    also \((\epsilon,l)\)-uniformly.
\end{lemma}
The proofs of~\Cref{prop Psi ineq} and~\Cref{lemma doubling} will be given in Section~\ref{sec master proof}
and Section~\ref{sec:redin}, respectively.

\subsection{Proof of the bounds on \(\Psi^{\av/\iso}\) in~\Cref{final lemma,a priori lemma}}\label{sec lemmata}
\begin{proof}[Proof of~\Cref{a priori lemma}]
    Within the proof we repeatedly appeal to a simple argument we call \emph{iteration}. By this we mean 
    the following procedure. Fix an $\epsilon>0$. Suppose that for any $l\in \N$
     whenever \(X\prec x\) holding \((\epsilon,l)\)-uniformly implies
    \begin{equation}\label{X iter}
        X\prec A + \frac{x}{B} + x^{1-\alpha} C^{\alpha},
    \end{equation}
    \((\epsilon,l+l')\)-uniformly for some constants \(l'\in\N\), \(B\ge N^\delta\), \(A,C>0\), and 
    exponent \(0<\alpha<1\), and we know that \(X\prec N^D\) \((\epsilon,1)\)-uniformly initially
     (here $\delta, \alpha$ and $D$ are $N$-independent 
    positive constants,
    other quantities  may depend on $N$).
    Then  by iterating~\eqref{X iter} finitely many  times  (depending only  on $\delta, \alpha$ and $D$)   we arrive at 
    \begin{equation}\label{XAC}
        X\prec A + C
    \end{equation} 
   \((\epsilon,1+Kl')\)-uniformly. 
    Here \(K\) may depend on \(K=K(\delta,\alpha, D)\) but does not depend on \(\epsilon\). In our application \(B\ge(N\eta)^{1/4}\) and therefore \(\delta\) is practically some order one parameter depending only on the fixed \(\epsilon\) in~\cref{thm local law}.
    
    The proof of~\Cref{a priori lemma}  is a two-step induction on $k$. 
    Our first step is to  establish the induction hypothesis
    \begin{equation}\label{psi12 claim}
        \Psi_{1}^{\av/\iso}\prec 1\le\sqrt{N\eta},\quad \Psi_2^{\av/\iso}\prec\sqrt{N\eta},
    \end{equation}
    \(\epsilon'\)-uniformly for some \(\epsilon'>0\). In fact for $ \Psi_{1}^{\av/\iso}$ we will establish the stronger $\prec 1$ bound immediately. 
   We assume that for some $l\in \N$ \[\Psi_k^{\av/\iso}\prec\psi_k^{\av/\iso}\] \((\epsilon,l)\)-uniformly initially, Then~\cref{eq Psi2av} together with~\cref{eq doubling} implies
    \begin{equation}\label{eq Psi2av 0th step}
        \begin{split}
            \Psi_2^\av &\prec 1+\psi_1^\av+\frac{\psi_{2}^\iso + (\psi_2^\av)^{1/2}+(\psi_4^\av)^{1/2}}{\sqrt{N\eta}}+\frac{(\psi_1^\iso)^2+(\psi_1^\av)^2+\psi_1^\iso(\psi_2^\av)^{1/2}}{N\eta}\\
            &\prec \sqrt{N\eta} + \psi_1^\av + \frac{\psi_{2}^\iso + \psi_2^\av}{\sqrt{N\eta}}+\frac{(\psi_1^\iso)^2+(\psi_1^\av)^2+\psi_1^\iso(\psi_2^\av)^{1/2}}{N\eta}        
        \end{split}
    \end{equation}
    \((\epsilon,l+1)\)-uniformly and hence, using iteration and a Schwarz inequality  $\psi_1^\iso(\psi_2^\av)^{1/2}\le (\psi_1^\iso)^2+\psi_2^\av$ for the last term,
    we get
    \begin{equation}
    \label{eq Psi2av 1st step}
        \begin{split}
            \Psi_2^\av &\prec \sqrt{N\eta} + \psi_1^\av + \frac{\psi_{2}^\iso}{\sqrt{N\eta}}+\frac{(\psi_1^\iso)^2+(\psi_1^\av)^2}{N\eta},
        \end{split}
    \end{equation}
    again \((\epsilon,l+1)\)-uniformly. 
    Next, we consider~\cref{eq Psi2iso} and   eliminate $\psi_3^\iso, \psi_4^\iso$ from it
    by first using~\cref{eq doubling,eq doubling iso} 
    in the form 
    \begin{equation}
        \begin{split}
            \Psi_3^\iso &\prec N\eta \Bigl(1+\frac{\psi_2^\iso}{\sqrt{N\eta}}\Bigr)\Bigl(1+\frac{\psi_2^\av}{N\eta}\Bigr)^{1/2}\\
            &\prec N\eta \Bigl(1+\frac{\psi_2^\iso}{\sqrt{N\eta}}\Bigr)\Bigl(1+\frac{\psi_2^\iso}{(N\eta)^{3/2}}+\frac{(\psi_1^\av)^{2}+(\psi_1^\iso)^{2}}{(N\eta)^{2}}\Bigr)^{1/2},  \\
            \Psi_4^\iso &\prec  N\eta\Bigl(1+\frac{\psi_2^\iso}{\sqrt{N\eta}}\Bigr)\Bigl(1+\frac{\psi_4^\av}{N\eta}\Bigr)^{1/2}\prec (N\eta )^{3/2}\Bigl(1+\frac{\psi_2^\iso}{\sqrt{N\eta}}\Bigr)\Bigl(1+\frac{\psi_2^\av}{N\eta}\Bigr)\\
            &\prec (N\eta )^{3/2}\Bigl(1+\frac{\psi_2^\iso}{\sqrt{N\eta}}\Bigr)\Bigl(1+\frac{\psi_2^\iso}{(N\eta)^{3/2}}+\frac{(\psi_1^\av)^{2}+(\psi_1^\iso)^{2}}{(N\eta)^{2}}\Bigr),
        \end{split}
    \end{equation}
    \((\epsilon,l+1)\)-uniformly, where in the second step we also eliminated $\psi_2^\av$ using~\cref{eq Psi2av 1st step}.   
    Plugging these bounds into~\cref{eq Psi2iso}
    yields
    \begin{equation}\label{eq Psi2iso 0th step}
        \begin{split}
            \Psi_2^\iso&\prec 1+\psi_1^\iso+\frac{\psi_2^\iso+(\psi_1^\iso\psi_3^\iso)^{1/2}+\psi_1^\av}{\sqrt{N\eta}}+\frac{\psi_1^\iso\psi_1^\av}{N\eta}+\frac{(\psi_3^\iso)^{1/2}+(\psi_4^\iso)^{1/2}}{(N\eta)^{1/4}}\\
            &\prec \sqrt{N\eta} + \psi_1^\iso + \frac{\psi_2^\iso}{\sqrt{N\eta}} + \frac{(\psi_1^\iso)^2+(\psi_1^\av)^2}{N\eta} +\sqrt{\psi_2^\iso}(N\eta)^{1/4}\Bigl(1+\frac{\psi_1^\av+\psi_1^\iso}{N\eta}\Bigr). %
        \end{split}
    \end{equation}
    \((\epsilon,l+2)\)-uniformly. By iteration we thus obtain 
    \begin{equation}\label{eq Psi2iso 1st step}
        \begin{split}
            \Psi_2^\iso&\prec \sqrt{N\eta}+\psi_1^\iso + \frac{(\psi_1^\iso)^2+(\psi_1^\av)^2}{N\eta},
        \end{split}
    \end{equation}
    \((\epsilon,l+K)\)-uniformly
    and by feeding~\cref{eq Psi2iso 1st step} back into~\cref{eq Psi2av 1st step} we conclude 
    \begin{equation}\label{psi2av step}
        \Psi_2^\av \prec \sqrt{N\eta} + \psi_1^\av + \frac{(\psi_1^\iso)^2+(\psi_1^\av)^2}{N\eta}
    \end{equation}
    \((\epsilon,l+K')\)-uniformly. By using~\eqref{eq Psi2iso 1st step} in~\cref{eq Psi1iso} we immediately obtain
    \begin{equation}\label{psi1iso step}
        \Psi_1^\iso \prec 1 + \frac{\psi_1^\iso+\psi_1^\av}{\sqrt{N\eta}} + \frac{\psi_1^\iso+\psi_1^\av}{(N\eta)^{3/4}}\prec 1 + \frac{\psi_1^\av}{(N\eta)^{1/2}}
    \end{equation}
   \((\epsilon,l+K'')\)-uniformly and together with~\cref{eq Psi1av} we also have that 
    \begin{equation}\label{psi1av step}
        \Psi_1^\av \prec 1 + \frac{(\psi_2^\av)^{1/2}}{\sqrt{N\eta}}
    \end{equation}
    \((\epsilon,l+K''')\)-uniformly. 
    Finally, by combining~\cref{psi2av step,psi1iso step,psi1av step} we obtain 
    \begin{equation}
        \Psi_2^\av\prec\sqrt{N\eta} + \frac{(\psi_2^\av)^{1/2}}{\sqrt{N\eta}} +  \frac{\psi_2^\av}{(N\eta)^2}\prec\sqrt{N\eta}
    \end{equation}
    and therefore \(\Psi_1^{\av/\iso}\prec 1\) \((\epsilon,l+K'''')\)-uniformly and finally, by~\eqref{eq Psi2iso 1st step}, all statements  
    in the claim~\cref{psi12 claim} hold for \(\epsilon'=\epsilon/2\) uniformly to absorb the factor \(K''''\).
    This  completes the initial step of the induction. In the sequel we refrain from specifying the precise \((\epsilon,l)\)-uniformity since in the end \(\epsilon\) can be chosen arbitrarily small and we only use~\cref{prop Psi ineq} finitely many often.
    
    Now we turn to the induction step: we assume that \(k\ge 4\) is even and that the bounds
    \begin{equation}\label{ind}
        \Psi_{n}^{\av/\iso}\prec \sqrt{N\eta}, \qquad n\le k-2
    \end{equation}
    have already been proved. We will prove the same bounds for $n=k-1, k$.
    
    For any \(j\le k\) and under the assumption~\eqref{ind}
    the reduction inequalities~\cref{eq doubling,eq doubling iso} simplify (recall that \(k\) is even) to
    \begin{equation}\label{j iso doubling}
        \begin{split}
            \Psi_{k+j}^\iso &\prec N\eta\Bigl(1+\frac{\psi_k^\iso}{\sqrt{N\eta}}\Bigr)\Bigl(1+\frac{\psi_{2j}^\av}{N\eta}\Bigr)^{1/2}\\ 
            &\prec N\eta\Bigl(1+\frac{\psi_k^\iso}{\sqrt{N\eta}}\Bigr)\begin{cases}
                \sqrt{N\eta}+ \frac{\psi_j^\av}{\sqrt{N\eta}}, & j\text{ even},\\
                \sqrt{N\eta} + \sqrt{\psi_{j-1}^\av+\psi_{j+1}^\av+\psi_{j-1}^\iso\psi_{j+1}^\av/N\eta}, & j\text{ odd}, 
            \end{cases} \\
            &\prec (N\eta)^{3/2}\Bigl(1+\frac{\psi_k^\iso}{\sqrt{N\eta}}\Bigr)\Bigl(1+ \frac{\psi_k^\av}{N\eta}\Bigr)^{\bm 1(j=k)+\bm1(j=k-1)/2} \prec (N\eta)^{3/2}\Bigl(1+\frac{\psi_k^\iso}{\sqrt{N\eta}}\Bigr)\Bigl(1+ \frac{\psi_k^\av}{N\eta}\Bigr)
        \end{split}
    \end{equation}
    and 
    \begin{equation}\label{j doubling}
        \Psi_{2j}^\av \prec   \left.\begin{cases}(N\eta)^2+
            (\psi_k^\av)^{2},& j=k,\\
            (N\eta)^2+N\eta \psi_k^\av,&  j=k-1,\\
            (N\eta)^2,&\text{else},
        \end{cases}\right\}\prec (N\eta)^2 +  (\psi_k^\av)^{2}.
    \end{equation}
    Then together with~\cref{eq Psikav,eq Psikiso} it follows that 
    \begin{equation}\label{Psik-1aviso 1st step}
        \begin{split}
            \Psi_{k-1}^\av &\prec \sqrt{N\eta} +\frac{\psi_{k-1}^\iso+\sum_{j=k/2}^{k-1} (\psi_{2j}^\av)^{1/2}}{\sqrt{N\eta}}\prec  \sqrt{N\eta} +\frac{\psi_{k-1}^\iso+\psi_k^\av}{\sqrt{N\eta}}  \\
            \Psi_{k-1}^\iso&\prec \sqrt{N\eta} + \frac{\sum_{j=k}^{2k-2}(\psi_j^\iso)^{1/2}}{(N\eta)^{1/4}}\prec \sqrt{N\eta} \Bigl(1+\frac{\psi_k^\iso}{\sqrt{N\eta}}\Bigr)^{1/2}\Bigl(1+\frac{\psi_k^\av}{N\eta}\Bigr)^{1/2}
        \end{split}
    \end{equation}
    and
    \begin{equation}\label{Psikaviso 1st step}
        \begin{split}
            \Psi_{k}^\av &\prec \sqrt{N\eta} + \psi_{k-1}^\av +\frac{\psi_{k-1}^\iso+\psi_{k}^\iso+\sum_{j=k/2}^k (\psi_{2j}^\av)^{1/2}}{\sqrt{N\eta}} \\
            &\prec \sqrt{N\eta} + \psi_{k-1}^\av +\frac{\psi_{k-1}^\iso+\psi_{k}^\iso+\psi_k^\av}{\sqrt{N\eta}} \\
            \Psi_{k}^\iso&\prec \sqrt{N\eta}  + \psi_{k-1}^\iso+ \frac{\psi_{k-1}^\av+(\psi_{k-1}^\iso\psi_{k+1}^\iso)^{1/2}}{\sqrt{N\eta}} + \frac{\sum_{j=k+1}^{2k}(\psi_j^\iso)^{1/2}}{(N\eta)^{1/4}}\\
            &\prec \sqrt{N\eta}  + \psi_{k-1}^\iso+ \frac{\psi_{k-1}^\av+\psi_k^\iso}{\sqrt{N\eta}} + \sqrt{N\eta} \Bigl(1+\frac{\psi_k^\iso}{\sqrt{N\eta}}\Bigr)^{1/2} \Bigl(1+\frac{\psi_k^\av}{N\eta}\Bigr)^{1/2}%
        \end{split}
    \end{equation} 
    where we used the first inequality of~\cref{j iso doubling} to estimate $\psi_{k+1}^\iso$ in the \(\sqrt{\psi_{k-1}^\iso\psi_{k+1}^\iso}\)-term with \(\psi_2^\av=\sqrt{N\eta}\). Iterating~\cref{Psikaviso 1st step} yields 
    \begin{equation}\label{Psikaviso 1.5st step}
        \begin{split}
            \Psi_{k}^\av &\prec \sqrt{N\eta} + \psi_{k-1}^\av +\frac{\psi_{k-1}^\iso+\psi_{k}^\iso}{\sqrt{N\eta}} \\
            \Psi_{k}^\iso&\prec \sqrt{N\eta}  + \psi_{k-1}^\iso +  \frac{\psi_{k-1}^\av+\psi_k^\av}{\sqrt{N\eta}},
        \end{split}
    \end{equation} 
    and by using~\cref{Psik-1aviso 1st step} in~\eqref{Psikaviso 1.5st step} it follows that 
    \begin{equation}\label{Psikaviso 2nd step}
        \begin{split}
            \Psi_{k}^\av &\prec \sqrt{N\eta} + \frac{\psi_{k}^\iso+\psi_k^\av}{\sqrt{N\eta}}+\Bigl(1+\frac{\psi_k^\iso}{\sqrt{N\eta}}\Bigr)^{1/2}\Bigl(1+\frac{\psi_k^\av}{N\eta}\Bigr)^{1/2}\\
            &\prec \sqrt{N\eta} + \frac{\psi_{k}^\iso}{\sqrt{N\eta}}\\
            \Psi_k^\iso &\prec \sqrt{N\eta} \Bigl(1+\frac{\psi_k^\iso}{\sqrt{N\eta}}\Bigr)^{1/2}\Bigl(1+\frac{\psi_k^\av}{N\eta}\Bigr)^{1/2} +  \frac{\psi_k^\av}{\sqrt{N\eta}}\\
            &\prec \sqrt{N\eta}+\sqrt{\psi_k^\av} + \frac{\psi_k^\av}{\sqrt{N\eta}}\prec \sqrt{N\eta}+ \frac{\psi_k^\av}{\sqrt{N\eta}}.
        \end{split}
    \end{equation}
    From~\cref{Psikaviso 2nd step} we immediately conclude \(\Psi_k^{\av/\iso}\prec\sqrt{N\eta}\) and by feeding this back into~\cref{Psikaviso 1st step} finally that 
    \begin{equation}
        \Psi_{k-1}^{\av/\iso}+\Psi_{k}^{\av/\iso}\prec\sqrt{N\eta},
    \end{equation}
    concluding the induction step. 
\end{proof}
\begin{proof}[Proof of~\Cref{final lemma}]
    This follows directly from~\Cref{a priori lemma,prop Psi ineq} and induction on $k$.
\end{proof}

\section{Proof of the master inequalities, \Cref{prop Psi ineq}}\label{sec master proof}
We recall the definition of the \emph{second order renormalisation}, denoted by underlining,
from~\cite{MR4334253}. For functions \(f(W),g(W)\) of the random matrix \(W\) we define 
\begin{equation}\label{underline}
    \underline{f(W)Wg(W)}:=f(W)Wg(W)-\E_{\widetilde W     } \Bigl[ (\partial_{\widetilde W} f)(W) \widetilde W g(W)+ f(W) \widetilde W (\partial_{\widetilde W} g)(W)\Bigr],
\end{equation}
where \(\partial_{\wt W}\) denotes the directional derivative in the direction of a GUE matrix \(\wt W\)
that is independent of $W$. The expectation is w.r.t.\ this GUE matrix.  Note that if $W$ itself is a GUE matrix, 
then $\E  \underline{f(W)Wg(W)}=0$, while for $W$ with a general distribution  this expectation
is independent of the first two moments of $W$;
in other words the underline renormalises  $f(W)Wg(W)$ up to second order.  
We note that underline in~\eqref{underline}  is a well-defined notation 
only when the position of the ``middle'' \(W\)  to which the renormalisation refers
is unambiguous.  This is the case in all of our proof since \(f,g\) will be products of resolvents  not  explicitly 
involving monomials of $W$.

We also note that the directional derivative of the resolvent is given by 
\begin{equation}
    \partial_{\widetilde W} G = - G \widetilde W G,
\end{equation}
furthermore, we have
\begin{equation}
    \E_{\widetilde W } \widetilde W A \widetilde W = \braket{A}\cdot I. 
\end{equation}  
For example, in case of $f= I$ and $g(W)= (W-z)^{-1}=G$ we have 
\[
\un{WG} = WG + \braket{G}G.
\]
Similarly, for $G_i= G(z_i)$ we also have
\[ \un{WG_1 G_2} =  WG_1G_2 + \braket{G_1} G_1 G_2 + \braket{G_1 G_2} G_2, \quad
\un{G_1 WG_2} =  G_1W G_2 + \braket{G_1} G_1 G_2 + \braket{G_2} G_1 G_2
\]
indicating that the definition of the underline in~\eqref{underline} depends on the "left" and "right" functions
$f$ and $g$, and even though $f(W) W g(W) = W f(W) g(W) = f(W) g(W) W$, their second order renormalisations
are not the same.

Using this underline notation and the defining equation for \(m=m_\mathrm{sc}\), we have 
\begin{equation}
    \label{eq:resG}
    G = m - m\un{WG} + m \braket{G-m} G = m - m\un{GW} + m \braket{G-m} G.
\end{equation}

The key idea of the proof of  \Cref{prop Psi ineq}
 is using~\cref{eq:resG} for some \(G_j\) in \(G_1A_1\ldots A_{k-1}G_k\) and extending the renormalisation to the whole product at the expense adding  resolvent products of  lower order.  For example, 
\begin{equation}\label{3G exp ex}
    \begin{split}
        &G_1A_1G_2A_2G_3\Bigl(1+\landauO*{\frac{1}{N\eta}}\Bigr)  \\
        &= m_2 G_1 A_1 A_2 G_3 - m_2 G_1A_1 \un{WG_2}A_2G_3 \\
        &= m_2\Bigl( G_1 A_1 A_2 G_3 + \braket{G_1A_1}G_1G_2A_2G_3 + G_1A_1G_3\braket{G_2 A_2 G_3} - \un{G_1A_1 WG_2A_2G_3}\Bigr),
    \end{split}
\end{equation}
where on the rhs.\ only products of resolvent with one deterministic matrix need to be understood. 
 The renormalisation of the whole product will be handled by cumulant expansion exploiting that its expectation
vanishes up to second order. 
We note that while \(\un{WG}=\un{GW}\), replacing \(G_2\) by \(m_2-m_2\un{G_2W}\) instead of \(m_2-m_2\un{WG_2}\) 
in~\eqref{3G exp ex} still gives a slightly different expression:
\begin{equation}\label{3G exp ex2}
    \begin{split}
        &G_1A_1G_2A_2G_3\Bigl(1+\landauO*{\frac{1}{N\eta}}\Bigr)  \\
        &= m_2 G_1 A_1 A_2 G_3 - m_2 G_1A_1 \un{G_2W}A_2G_3 \\
        &= m_2 \Bigl(G_1 A_1 A_2 G_3 + \braket{G_1A_1G_2}G_1A_2G_3 + G_1A_1G_2G_3\braket{A_2 G_3} - \un{G_1A_1 G_2WA_2G_3}\Bigr).
    \end{split}
\end{equation}
A key ingredient for the proof is the following lemma which shows that the deterministic approximation \(M\) defined in~\cref{eq M def} satisfies the same recursive relations as suggested by~\cref{3G exp ex,3G exp ex2}
after ignoring the full underline term and the $1/(N\eta)$ error terms.

\begin{lemma}\label{lemma M rec}
    Let \(z_1,\ldots,z_k\) by spectral parameters, and \(A_1,\ldots,A_{k-1}\) be deterministic matrices. Then for any \(1\le j\le k\) we have the relations 
    \begin{equation}\label{M rec 1}
        \begin{split}
            M(z_1,\ldots,z_k) &= m_j M(z_1,\ldots,z_{j-1},A_{j-1}A_j,z_{j+1},\ldots,z_k) \\
            &\quad + m_j\sum_{l=1}^{j-1} M(z_1,\ldots,A_{l-1},z_l,I,z_j,A_j,\ldots,z_k) \braket{M(z_l,A_l,\ldots,z_{j-1})A_{j-1}}\\
            &\quad + m_j\sum_{l=j+1}^k M(z_1,\ldots,A_{j-1},z_l,A_l,\ldots,z_k)\braket{M(z_j,A_j,\ldots,z_{l})},
        \end{split}
    \end{equation}
    and 
    \begin{equation}\label{M rec 2}
        \begin{split}
            M(z_1,\ldots,z_k) &= m_j M(z_1,\ldots,z_{j-1},A_{j-1}A_j,z_{j+1},\ldots,z_k) \\
            &\quad + m_j\sum_{l=1}^{j-1} M(z_1,\ldots,A_{l-1},z_l,A_j,\ldots,z_k) \braket{M(z_l,A_l,\ldots,z_{j})}\\
            &\quad + m_j\sum_{l=j+1}^k M(z_1,\ldots,A_{j-1},z_j,I, z_l,A_l\ldots,z_k)\braket{M(z_l,A_j,\ldots,z_{l-1})A_{l-1}}.
        \end{split}
    \end{equation}
\end{lemma}
 We remark that the special $j=1$ case of this lemma was already proven in~\cite[Lemma 5.4]{MR4372147}.
We will present a direct combinatorial proof for the general case %
in~\cref{sec:addproofs}.  Alternatively, \cref{lemma M rec}  can also be deduced
from the original expansions for  resolvent products with the full underline term.
For example, taking the expectation 
of \eqref{3G exp ex} for $W$ being a GUE matrix and letting $N\to\infty$
removes the full underline term and the error terms. 
Since the local law~\cite[Theorem 3.4]{MR4372147} asserts that $G_1A_1G_2A_2G_3$
asymptotically  equals $M(z_1, A_1, z_2, A_2, z_3)$ in the $N\to\infty$ limit
for any fixed spectral parameters, we obtain the corresponding
identity~\eqref{M rec 1} for $k=3$. The argument for general $k$ is identical.

\subsection{Proof of~\Cref{prop Psi ineq}} 
The proofs of the averaged and isotropic bounds are done separately below. For simplicity we do not carry the dependence on the spectral parameters \(z_j\) and traceless matrices \(A_j\) but instead simply write \(G\) and \(A\).
\subsubsection{Averaged bounds~\cref{eq Psi1av,eq Psi2av,eq Psikav}}\label{sec:avve}
Within the proof we repeatedly make use of the a priori bounds~\cref{eq a priori 4,eq a priori k} for \(j\le 2k\). It is important to stress that 
after possibly applying Lemma~\ref{G^2 lemma} no chains of length more than $2k$
arise along our expansion hence the a priori bounds are needed up to index $2k$ only.

By~\cref{eq:resG} for the first $G$ and using the local law \(\abs{\braket{G-m}}\prec 1/(N\eta)\) we obtain
\begin{equation}\label{eq av un repl0}
    \begin{split}
        &\braket{(GA)^k} \Bigl(1+\landauOprec*{(N\eta)^{-1}}\Bigr)\\
        &\quad= m \braket{A(GA)^{k-1}} - m \braket{\un{WGA}(GA)^{k-1}} \\
        &\quad = m \braket{A(GA)^{k-1}} + m\sum_{j=1}^{k-1}\braket{(GA)^j G}\braket{(GA)^{k-j}} - m\braket{\un{W(GA)^k}}.        
    \end{split}
\end{equation}
By assumption~\cref{eq a priori 4,eq a priori k} and~\Cref{G^2 lemma} we have 
\begin{equation}
    \begin{split}
        \abs*{\braket{A(GA)^{k-1}}- \braket{A M_{k-1}A}} &\prec \frac{\psi^\av_{k-2}}{N\eta^{k/2}}+\frac{\psi^\av_{k-1}}{N\eta^{k/2-1/2}}\lesssim  \frac{\psi^\av_{k-1}+\psi^\av_{k-2}}{N\eta^{k/2}}\\
        \abs{\braket{(GA)^jG-M_{j+1}}} &\prec \frac{\psi_j^\av}{N\eta^{j/2+1}}, \\ 
        \abs*{\braket{(GA)^{k-j}}-\braket{M_{k-j}A}}&\prec \frac{\psi_{k-j}^\av}{N\eta^{(k-j)/2}},
    \end{split}    
\end{equation}
so we can replace each resolvent chain by its deterministic $M$-value plus the error term.
In particular, for the middle term in the third line of~\eqref{eq av un repl0} by  a telescopic summation we have
\begin{equation}\label{telescope}
    \begin{split}
        &\abs*{\sum_{j=1}^k\Bigr(\braket{(GA)^j G}\braket{(GA)^{k-j}}-\sum_{j=1}^k\braket{M_{j+1}}\braket{M_{k-j}A}\Big)} \\
        & \prec   \sum_{j=2}^{k-1}\frac{1}{\eta^{\lfloor j/2\rfloor}}\frac{\psi_{k-j}^\av}{N\eta^{(k-j)/2}} + \sum_{j=1}^{k-4} \frac{1}{\eta^{\lfloor(k-j)/2\rfloor}}\frac{\psi_j^\av}{N\eta^{j/2}}+ \frac{\psi_{k-2}^\av}{N\eta^{k/2}}+\frac{\psi_{k-3}^\av}{N\eta^{k/2-1/2}} + \sum_{j=1}^{k-1}\frac{\psi_j^\av\psi_{k-j}^\av}{N^2\eta^{k/2+1}}\\
        &\lesssim \frac{1}{N\eta^{k/2}}\biggl(\psi_{k-2}^\av + \sum_{j=1}^{k-1}\psi_j^\av\Big(1+\frac{\psi_{k-j}^\av}{N\eta}\Bigr) \biggr),
    \end{split}
\end{equation}
where we used that by assumption \(\eta\lesssim 1\), the bounds~\cref{Msize} and \(\braket{M_{2}}=\braket{M_1A}=0\). Together with the deterministic identity~\cref{M rec 1} 
we conclude 
\begin{equation}\label{eq av un repl}
    \begin{split}
        \bigl(\braket{(GA)^k}-\braket{M_kA}\bigr)\Bigl(1+\landauOprec*{(N\eta)^{-1}}\Bigr) ={} - m\braket{\un{W(GA)^k}}  + \landauOprec*{\cE_k^\mathrm{av}}
    \end{split}
\end{equation}
with 
\begin{equation}
    \cE_k^\av := \frac{1}{N\eta^{k/2}}\biggl(1+\sum_{j=1}^{k-1}\psi_j^\av\Bigl(1+\frac{\psi_{k-j}^\av}{N\eta}\Bigr)\biggr),
\end{equation}
where we used \(\abs{\braket{M_2A}}\lesssim1\) and \(\abs{\braket{M_kA}}\lesssim\eta^{1-k/2}\) for \(k\ge3\). 

We recall the cumulant expansion 
\begin{equation}\label{eq cum exp2}
  \begin{split}
    \E w_{ab} f(W) &=  \E\frac{\partial_{ba} f(W) + \sigma \partial_{ab} f(W)}{N}+\sum_{k=2}^R\sum_{q+q'=k} \frac{\kappa^{q+1,q'}_{ab}}{N^{(k+1)/2}} \E \partial_{ab}^q \partial_{ba}^{q'} f(W) + \Omega_R,
  \end{split}
\end{equation}
from~\cite[Eq.~(79)]{MR4334253} with an error term \(\Omega_R\) which for the application in~\cref{eq av cum exp} below can be easily seen to be of size \(\Omega_R=\landauO{N^{-2p}}\) for \(R=12p\).
Here the first fraction represents the Gaussian contribution and \(\sigma=N \E w_{12}^2\in\set{0,1}\) is determined by the complex/real symmetry class of \(W\) due to~\Cref{def:wigner}. The sum in~\cref{eq cum exp2} represents the non-Gaussian contribution and \(\kappa^{p,q}_{ab}\) denotes the joint cumulant of \(p\) copies of \(\sqrt{N}w_{ab}\) and \(q\) copies of \(\sqrt{N}\ov{w_{ab}}\). 
Using~\cref{eq cum exp2,eq av un repl} and distributing the derivatives we obtain
\begin{equation}\label{eq av cum exp}
    \begin{split}
        &\E \abs{\braket{( GA)^k-M_k A }}^{2p} \\
        &= \abs*{-m \E  \braket{\un{W(GA)^k}} \braket{ (GA)^k-M_kA}^{p-1}\braket{ (G^\ast A)^k-M_k^\ast A}^p} + \landauOprec*{(\cE_k^\mathrm{av})^{2p}}\\
        &\lesssim \E |m|\frac{\abs{\braket{(GA)^{2k}G}}+\abs{\braket{(GA)^{k}(G^\ast A)^kG^\ast}}}{N^2} \abs*{\braket{(GA)^k-M_kA}}^{2p-2} \\
        &\quad+ \sum_{\abs{\vl}+\sum (J\cup J_\ast) \ge 2}  \E  \Xi_k^\mathrm{av}(\vl,J,J^\ast)\abs*{\braket{(GA)^k-M_kA}}^{2p-1-\abs{J\cup J_\ast}}+\landauOprec*{(\cE_k^\mathrm{av})^{2p}},
    \end{split}
\end{equation}
where \(\Xi_k^\mathrm{av}(\vl,J,J_\ast)\) is defined as
\begin{equation}\label{Xi av}
    \begin{split}
        \Xi^\mathrm{av}_k &:= |m| N^{-(\abs{\vl}+\sum (J\cup J_\ast)+3)/2}\sum_{ab} \abs{\partial^{\vl} ((GA)^k)_{ba}} \prod_{\vj\in J}\abs{\partial^\vj\braket{(GA)^k}}\prod_{\vj\in J_\ast}\abs{\partial^{\vj}\braket{(G^\ast A)^k}}, 
    \end{split}
\end{equation}
and the summation in~\cref{eq av cum exp} is taken over tuples \(\vl\in\Z_{\ge0}^2\) and multisets of tuples \(J,J_\ast\subset\Z_{\ge0}^2\setminus\set{0,0}\). Moreover, we set \(\partial^{(l_1,l_2)}:=\partial_{ab}^{l_1}\partial_{ba}^{l_2}\), \(\abs{(l_1,l_2)}=l_1+l_2\) and \(\sum J:=\sum_{\vj\in J}\abs{\vj}\). For the first term in the third line of~\eqref{eq av cum exp} we have 
\begin{equation}\label{av Gaussian bound}
    \begin{split}
        |m|\frac{\abs{\braket{(GA)^{2k}G}}+\abs{\braket{(GA)^{k}(G^\ast A)^kG^\ast}}}{N^2} \prec  \frac{1}{N^2\eta^{k}}\Bigl(1+\frac{\psi_{2k}^\mathrm{av}}{N\eta}\Bigr)
    \end{split}
\end{equation}
due to~\Cref{G^2 lemma} and \(\abs{\braket{M_{2k+1}}}\lesssim \eta^{-k}\) from~\eqref{Msize}. 

We now turn to the estimate on \(\Xi^\mathrm{av}\) from~\cref{Xi av}. Due to the Leibniz rule the derivatives can be written as a sum of products of \((aa,bb,ab,ba)\)-entries of resolvent chains of the form \(GAG\cdots AG(A)\), e.g.\ 
\begin{equation}
    \begin{split}
        \partial_{ab}\partial_{ba}(GAGA)_{ba}  &= G_{ba}G_{bb}(GAGA)_{aa}+G_{bb}G_{aa}(GAGA)_{ba}+G_{bb}(GAG)_{aa}(GA)_{ba} \\
        &\quad + G_{ba}(GAG)_{bb}(GA)_{aa}+(GAG)_{ba} G_{bb}(GA)_{aa}+(GAG)_{bb}G_{aa}(GA)_{ba}\\
        \partial_{ba}\partial_{ab}\braket{GA} &= \frac{G_{bb}(GAG)_{aa}+(GAG)_{bb}G_{aa}}{N}.
    \end{split}
\end{equation}
Thus we have the \emph{naive bounds} 
\begin{equation}\label{eq:naiveb}
    \begin{split}
        \abs{\partial^{\vl} ((GA)^k)_{ba}} &\prec \frac{1}{\eta^{(k-1)/2}}\sum_{k_0+\cdots+k_{\abs{\vl}}=k-1}\prod_i\Bigl(1+\frac{\psi_{k_i}^\mathrm{iso}}{\sqrt{N\eta}}  \Bigr)\prec \frac{1}{\eta^{(k-1)/2}}\Bigl(1 + \frac{\psi_{k-1}^\iso}{\sqrt{N\eta}}\Bigr),\\
        \abs{\partial^{\vj} \braket{(G^{(\ast)}A)^k}} &\prec \frac{1}{N\eta^{k/2}}\sum_{k_1+\cdots+k_{\abs{\vj}}=k}\prod_i\Bigl(1+\frac{\psi_{k_i}^\mathrm{iso}}{\sqrt{N\eta}}  \Bigr)\prec \frac{1}{N\eta^{k/2}} \Bigl(1 + \frac{\psi_{k}^\iso+\psi_{k-1}^\iso}{\sqrt{N\eta}}\Bigr),
    \end{split}
\end{equation}
where we used that $\psi_{k_i}^\iso=\sqrt{N\eta}$ for $k_i \le k-2$ by~\eqref{eq a priori k}
by assumption. In the proof of the bounds~\eqref{eq:naiveb} we used that
\begin{equation}
    \label{eq:addandsub}
    \abs*{( (GA)^{k_i})_{ab}}=\abs*{(M_{k_i}A)_{ab}+\landauOprec*{\frac{\psi_{k_i}^{\iso}}{\sqrt{N\eta^{k_i+1}}}}}\prec \frac{1}{\eta^{k_i/2}}\left(1+\frac{\psi_{k_i}^{\iso}}{\sqrt{N\eta}}\right),
\end{equation}
by~\eqref{Psi def} and the norm bound in~\eqref{Msize} for the deterministic term. 
We will use~\eqref{eq:naiveb} for any $k\ne 2$, the $k=2$ case will be done slightly differently later.

For $k\ne 2$, by~\eqref{eq:naiveb} we obtain 
\begin{equation}\label{Xi av bound without Ward}
    \begin{split}
        \abs{\Xi_k^\mathrm{av}} &\prec N^{(2-\abs{\vl}-\sum (J\cup J_\ast))/2}  \frac{\sqrt{N\eta}}{N\eta^{k/2}}\Bigl(\frac{1}{N\eta^{k/2}}\Bigr)^{\abs{J\cup J_\ast}} \biggl(1+\frac{\psi_{k-1}^\iso}{\sqrt{N\eta}} \biggr) \biggl(1 + \frac{\psi_{k}^\iso+\psi_{k-1}^\iso}{\sqrt{N\eta}}\biggr)^{\abs{J\cup J_\ast}}. %
    \end{split}
\end{equation}
Note that  estimating $\Xi_k^\mathrm{av}$ is necessary only if \(\abs{\vl}+\sum (J\cup J_\ast)\ge 2\)
by~\eqref{eq av cum exp},
so the $N$-prefactor in~\eqref{Xi av bound without Ward}
comes with a non-positive power. In fact, if \(\abs{\vl}+\sum (J\cup J_\ast)\ge 3\), then this factor
removes  the $\sqrt{N\eta}$ factor from the numerator, which will be sufficient for our purpose.

In case \(\abs{\vl}+\sum (J\cup J_\ast)=2\) we still wish to remove the $\sqrt{N\eta}$ factor,
so we need to improve~\eqref{Xi av bound without Ward}. We use a standard procedure, called the \emph{Ward improvement},
which relies on the fact that sums of the form $\sum_{ab} ((GA)^n G)_{ab}$  can be estimated more efficiently 
then just estimating each term one by one.
Note that in~\cref{Xi av}, after distributing the derivatives according to the Leibniz rule, necessarily some resolvent chain\footnote{Here the \([A]\) in square brackets indicates an optional matrix \(A\) which may or may not be present.} 
\((GAG\cdots AG[A])\) 
appears with off-diagonal indices \((a,b)\) or \((b,a)\). Indeed, an off-diagonal term comes from  one of the products in~\eqref{Xi av} when 
\(\abs{\vj}=1\) for some \(\vj\in J\cup J_\ast\), and it comes from  the $\partial^\vl ((GA)^k)_{ba}$ factor when  \(\abs{\vl}=0\) 
or \(\abs{\vl}=2\), by parity considerations.
For such off-diagonal resolvent chains we use 
\begin{equation}\label{ex with Ward}
    \begin{split}
        \abs*{\sum_{ab}(G[A])_{ab}} &\le N\Big( \sum_{ab} \abs{(G[A])_{ab}}^2\Big)^{1/2}= N^{3/2} \sqrt{\braket{G[A^2]G^\ast}}\\ 
        &\le N^{3/2} [\norm{A}]\sqrt{\braket{GG^\ast}} \prec  \frac{N^{3/2}}{\eta^{1/2}}\Bigl(1+\frac{(\psi_0^\av)^{1/2}}{\sqrt{N\eta}}\Bigr)\\
        \sum_{ab} \abs{((GA)^n G[A])_{ab}} &\le N^{3/2}\sqrt{\braket{(GA)^{n}G[A^2]G^\ast (AG^\ast)^{n}  }} \\
        &\le [\norm{A}] N^{3/2}\sqrt{\braket{(GA)^{n}GG^\ast (AG^\ast)^{n}  }}\prec \frac{N^{3/2}}{\eta^{(n+1)/2}}\biggl( 1+ \sqrt{\frac{\psi_{2n}^\mathrm{av}}{N\eta}} \biggr)
    \end{split}
\end{equation}
for \(n\ge 1\).
This allows us to gain a factor of \((N\eta)^{-1/2}\) 
compared with the naive bounds
\begin{equation}
    \begin{split}
        \abs*{\sum_{ab}(G[A])_{ab}} &\prec N^2 \Bigl(1+\frac{\psi_0^\iso}{\sqrt{N\eta}}\Bigr)\\
        \sum_{ab} \abs{((GA)^n G[A])_{ab}} &\prec \frac{N^2}{\eta^{n/2}} \Bigl(1+\frac{\psi_n^\mathrm{iso}}{\sqrt{N\eta}} \Bigr).
    \end{split}
\end{equation}
that were used in~\cref{Xi av bound without Ward}, at the expense
at the expense of replacing \(1+\psi_n^\iso/\sqrt{N\eta}\) by \(1+\sqrt{\psi_{2n}^\av/N\eta}\). Thus, in case \(\abs{\vl}+\sum (J\cup J_\ast)=2\) we can also improve upon~\cref{Xi av bound without Ward} by a factor of \((N\eta)^{-1/2}\) and obtain 
\begin{equation}\label{Xi av bound with Ward}
    \begin{split}
        \abs{\Xi_k^\mathrm{av}} &\prec  \Bigl(\frac{1}{N\eta^{k/2}}\Bigr)^{1+\abs{J\cup J^\ast}} \biggl(1+\frac{\psi_{k-1}^\iso+\psi_{k}^\iso + \sum_{j=\lfloor k/2\rfloor}^{k}(\psi_{2j}^\av)^{1/2}}{\sqrt{N\eta}}\biggr)^{\abs{J\cup J^\ast}+1},
    \end{split}
\end{equation}
where we used that  \(\psi_{2j}^\av=\sqrt{N\eta}\) for \(j< \lfloor k/2\rfloor\) 
from~\eqref{eq a priori k}. Combining  this with the earlier discussed \(\abs{\vl}+\sum (J\cup J_\ast)\ge 3\) case, we
obtain~\eqref{Xi av bound with Ward} for all cases.
By plugging~\cref{av Gaussian bound,Xi av bound with Ward} 
into~\cref{eq av cum exp}  we conclude  
\begin{equation}\label{Y1}
    \begin{split}
        &\E\abs{\braket{(GA)^k-M_kA}}^{2p} \prec (\cE_k^\av)^{2p} \\
        &+ \sum_{m=1}^{2p}\Bigl[\frac{1}{N\eta^{k/2}}\Bigl(1+\frac{\psi_{k-1}^\iso+\psi_{k}^\iso + \sum_{j=\lfloor k/2\rfloor}^{k}(\psi_{2j}^\av)^{1/2}}{\sqrt{N\eta}}\Bigr)\Bigr]^{m}\Bigl(\E\abs*{\braket{(GA)^k-M_kA}}^{2p}\Bigr)^{1-m/2p}
    \end{split}
\end{equation}
and get the appropriate estimate $\E\abs{ \cdots}^{2p}$ using Young inequalities. 
Since \(p\) is arbitrary, it follows that 
\begin{equation}\label{Y2}
    \begin{split}
        \abs{\braket{(GA)^k-M_kA}} &\prec  \cE_k^\av + \frac{1}{N\eta^{k/2}}\Bigl(1+\frac{\psi_{k-1}^\iso+\psi_{k}^\iso + \sum_{j=\lfloor k/2\rfloor}^{k}(\psi_{2j}^\av)^{1/2}}{\sqrt{N\eta}}\Bigr)\\
        &\prec \frac{1}{N\eta^{k/2}}\biggl(1+\sum_{j=1}^{k-1}\psi_j^\av+\frac{\psi_{k-1}^\iso+\psi_{k}^\iso + \sum_{j=\lceil k/2\rceil}^{k}(\psi_{2j}^\av)^{1/2}}{\sqrt{N\eta}}\biggr),
    \end{split}
\end{equation}
concluding the proof of~\cref{eq Psi1av,eq Psikav}.  Here we used that at least one factor in the $\psi_j^\av \psi_{k-j}^\av$ product
from  $\cE_k^\av$ is equal to $\sqrt{N\eta}$  by using~\eqref{eq a priori k},  since either $j$ or $k-j$ is smaller or equal than $k-2$ for \(k\ne 2\).

The proof of~\cref{eq Psi2av}, i.e.\ the \(k=2\) case, is identical except that in the second line of~\cref{eq:naiveb}
\begin{equation}\label{derfinal}
    \abs{\partial^j \braket{(G^{(\ast)}A)^2}} \prec \frac{1}{N\eta}\sum_{k_1+\cdots+k_{\abs{j}}=2}\prod_i\Bigl(1+\frac{\psi_{k_i}^\mathrm{iso}}{\sqrt{N\eta}}  \Bigr)\prec \frac{1}{N\eta^{k/2}} \Bigl(1 + \frac{\psi_{2}^\iso}{\sqrt{N\eta}}+\frac{(\psi_{1}^\iso)^2}{N\eta}\Bigr)
\end{equation}
and in \(\cE_2^\av\) there are quadratic terms resulting in \((\psi_1^\iso)^2,(\psi_1^\av)^2\) in~\cref{eq Psi2av}. 
This completes the estimates for the averaged quantities. 

\subsubsection{Isotropic bounds~\cref{eq Psi1iso,eq Psi2iso,eq Psikiso}}\label{sec:isso}
Similarly to~\eqref{eq av un repl0}, for the isotropic local law we start by comparing \(((GA)^k G-M_{k+1})_{\vx\vy}\) and \((\un{GAW(GA)^{k-1}G})_{\vx\vy}\) 
\begin{equation}\label{eq iso un repl0}
    \begin{split}
        &((GA)^kG)_{\vx\vy }\Bigl(1+\landauOprec*{(N\eta)^{-1}}\Bigr) \\
        &\quad= m (GA(AG)^{k-1})_{\vx\vy } - m (GA\un{WG}(AG)^{k-1})_{\vx\vy }\\
        &\quad= m (GA(AG)^{k-1})_{\vx\vy } - m (\un{GAWG(AG)^{k-1}})_{\vx\vy } +m \braket{GA}(G^2(AG)^{k-1})_{\vx\vy }\\
        &\qquad+ m\sum_{j=1}^{k-1} \braket{(GA)^jG} ((GA)^{k-j}G)_{\vx\vy }.
    \end{split}
\end{equation}
We again replace the $G$-chains with their deterministic counterparts using
\begin{equation}
    \begin{split}
        (GA(AG)^{k-1})_{\vx\vy } &= (M(z_1,A^2,z_3,\ldots))_{\vx\vy } + \landauOprec*{\frac{\psi_{k-1}^\iso}{\sqrt{N\eta^{k}}}+\frac{\psi_{k-2}^\iso}{\sqrt{N\eta^{k+1}}}}\\
        &=(M(z_1,A^2,z_3,\ldots))_{\vx\vy } + \landauOprec*{\frac{\psi_{k-1}^\iso+\psi_{k-2}^\iso}{\sqrt{N\eta^{k+1}}}}\\
        \abs*{\braket{GA}(G^2(AG)^{k-1})_{\vx\vy} }&\prec \frac{\psi_1^\av}{N\eta^{1/2}}\frac{1}{\eta^{(k+1)/2}}\Bigl(1+\frac{\psi_{k-1}^\iso}{\sqrt{N\eta}}\Bigr),
    \end{split}
\end{equation}
where we used the upper bound on $(G^2(AG)^{k-1})_{\vx\vy} $ from~\eqref{eq G^2 lemma}. 
By a telescopic replacement we have 
\begin{equation}
    \begin{split}
        &\abs*{\sum_{j=1}^{k-1} \Bigl(\braket{(GA)^jG} ((GA)^{k-j}G)_{\vx\vy }-\braket{M_{j+1}}(M_{k-j+1})_{\vx\vy}\Bigr)}\\
        &\prec \sum_{j=2}^{k-1} \abs{\braket{M_{j+1}}} \frac{\psi_{k-j}^\iso}{\sqrt{N\eta^{k-j+1}}} + \sum_{j=1}^{k-1} \frac{\psi_j^\av}{N\eta^{j/2+1}} \abs{(M_{k-j+1})_{\vx\vy}}+\sum_{j=1}^{k-1} \frac{\psi_j^\av}{N\eta^{j/2+1}}\frac{\psi_{k-j}^\iso}{\sqrt{N\eta^{k-j+1}}} \\
        &\lesssim \sum_{j=1}^{k-1}\frac{\psi_{k-j}^\iso}{\sqrt{N\eta^{k+1}}}\Bigl(1+\frac{\psi_j^\av}{N\eta}\Bigr) + \sum_{j=1}^{k-1} \frac{\psi_j^\av}{N\eta^{k/2+1}}.
    \end{split}
\end{equation}
and together with~\cref{M rec 1} we conclude from~\eqref{eq iso un repl0} that
\begin{equation}
    \begin{split}
        ((GA)^kG-M_{k+1})_{\vx\vy }\Bigl(1+\landauOprec*{(N\eta)^{-1}}\Bigr)&\quad= -m (\un{GAWG(AG)^{k-1}})_{\vx\vy } + \landauOprec*{\cE^\mathrm{iso}_k},
    \end{split}
\end{equation}
where 
\begin{equation}
    \cE^\mathrm{iso}_k:=\frac{1}{\sqrt{N}\eta^{(k+1)/2}} \Bigg(1+
    \sum_{j=1}^{k-1}\Big[\psi_{k-j}^\mathrm{iso}\Big(1+\frac{\psi_j^\mathrm{av}}{N\eta}\Big)+
    \frac{\psi_j^\mathrm{av}}{\sqrt{N\eta}}\Big]+\bm1(k=1)\frac{\psi_1^\mathrm{av}}{\sqrt{N \eta}} \Bigg).
\end{equation}

Thus, 
\begin{equation}\label{eq iso cum exp}
    \begin{split}
        &\E \abs{(( GA)^kG-M_{k+1} )_{\vx\vy }}^{2p} \\
        &= \abs*{-m \E  (\un{GAWG(AG)^{k-1}})_{\vx\vy } ( (GA)^kG-M_{k+1})_{\vx\vy }^{p-1}( (G^\ast A)^kG^\ast -M_{k+1}^\ast )_{\vy\vx }^p}\\
        &\qquad + \landauOprec*{(\cE_k^\mathrm{iso})^{2p}}\\
        &\lesssim \E\wt\Xi_k^\mathrm{iso} \abs*{((GA)^kG-M_{k+1})_{\vx\vy }}^{2p-2} +\landauOprec*{(\cE_k^\mathrm{iso})^{2p}}\\
        &\qquad + \sum_{\abs{\vl}+\sum (J\cup J_\ast) \ge 2}  \E \Xi_k^\mathrm{iso}(\vl,J,J_\ast)\abs*{((GA)^kG-M_{k+1})_{\vx\vy }}^{2p-1-\abs{J\cup J_\ast}},
    \end{split}
\end{equation}
where 
\begin{equation}\label{wt Xi iso}
    \begin{split}
        \wt\Xi_k^\mathrm{iso}&:=|m|\sum_{j=0}^{k} \biggl(\frac{\abs{(G(AG)^jG(AG)^{k-1})_{\vx\vy }  ( G(AG)^{k-j+1} )_{\vx\vy }}}{N} \\
        &\qquad\qquad + \frac{\abs{(G^\ast (AG^\ast)^jG(AG)^{k-1})_{\vy\vy }  ( GA(G^\ast A)^{k-j}G^\ast )_{\vx\vx }}}{N}\biggr)
    \end{split}
\end{equation}
and \(\Xi_k^\mathrm{iso}(\vl,J,J_\ast)\) is defined as
\begin{equation}\label{Xi iso}
    \begin{split}
        \Xi^\mathrm{iso}_k := |m|N^{-(\abs{\vl}+\sum (J\cup J_\ast)+1)/2}\sum_{ab}& \abs{\partial^\vl [(GA)_{\vx a}(G(AG)^{k-1})_{b\vy}]} \\
        & \times \prod_{\vj\in J}\abs{\partial^{\vj}((GA)^kG)_{\vx\vy }}\prod_{\vj\in J_\ast}\abs{\partial^{\vj}((G^\ast A)^kG^\ast)_{\vy\vx}} .
    \end{split}
\end{equation}
For~\cref{wt Xi iso} we estimate
\begin{equation}\label{iso Gaussian bound a priori}
    \begin{split}
        \wt\Xi_k^\iso &\prec \sum_{j=0}^k \Bigl(\norm{M_{k+j}} + \frac{\psi_{k+j-1}^\mathrm{iso}}{\sqrt{N}\eta^{(k+j)/2}}\Bigr)\Bigl(\norm{M_{k-j+2}} + \frac{\psi_{k-j+1}^\mathrm{iso}}{\sqrt{N}\eta^{(k-j+2)/2}}\Bigr) \\
        &\prec   \frac{1}{N\eta^{k+1}} \sum_{j=0}^k\Bigl(1+\frac{\psi_{k+j-1}^\iso}{\sqrt{N\eta}}\Bigr)\Bigl(1+\frac{\psi_{k-j+1}^\iso}{\sqrt{N\eta}}\Bigr).
    \end{split}
\end{equation}

In order to estimate \(\Xi^\mathrm{iso}_k\) we use the entrywise bounds
\begin{equation}\label{iso der bound}
    \begin{split}
        \abs{\partial^\vj ((G^{(\ast)}A)^kG^{(\ast)})_{\vx\vy }} &\prec \frac{1}{\eta^{k/2}}\sum_{k_0+\cdots+k_{\abs{\vj}}=k}\prod_i\Bigl(1+\frac{\psi_{k_i}^\mathrm{iso}}{\sqrt{N\eta}}  \Bigr)\\
        &\prec  \frac{1}{\eta^{k/2}}\Bigl(1 + \frac{\psi_{k}^\iso+\psi_{k-1}^\iso}{\sqrt{N\eta}}\Bigr)\\
        \abs{\partial^\vl (G(AG)^{k-1})_{b\vy}} & \prec \frac{1}{\eta^{(k-1)/2}}\sum_{k_0+\cdots+k_{\abs{\vl}+1}=k-1}\prod_i\Bigl(1+\frac{\psi_{k_i}^\mathrm{iso}}{\sqrt{N\eta}}  \Bigr)\\
        &\prec  \frac{1}{\eta^{(k-1)/2}}\Bigl(1+\frac{\psi_{k-1}^\mathrm{iso}}{\sqrt{N\eta}}  \Bigr).
    \end{split}
\end{equation}
Note that in the second step of the first inequality we tacitly assumed that \(k\ne 2\); the special case $k=2$ will be discussed
at the end of the proof. 
From~\eqref{iso der bound} we directly obtain  the naive bound
\begin{equation}\label{Xi iso bound with one Ward}
    \begin{split}
        \abs{\Xi^\mathrm{iso}_k} &\prec  \frac{\sqrt{N\eta}}{N^{(\abs{\vl}+\sum (J\cup J_\ast)-2)/2}} \Bigl(\frac{\sqrt{N\eta}}{\sqrt{N\eta^{k+1}}}\Bigr)^{1+\abs{J\cup J_\ast}} \Bigl(1+\frac{\psi_{k-1}^\mathrm{iso}}{\sqrt{N\eta}}  \Bigr)\Bigl(1 + \frac{\psi_{k}^\iso+\psi_{k-1}^\iso}{\sqrt{N\eta}}\Bigr)^{\abs{J\cup J_\ast}}.
    \end{split}
\end{equation}
Recalling the definition~\cref{Xi iso} and that we need to estimate $\Xi^\mathrm{iso}_k$ only when $\abs{\vl}+\sum (J\cup J_\ast) \ge 2$
by~\eqref{eq iso cum exp},
we claim that we can improve upon~\cref{Xi iso bound with one Ward} by 
\begin{enumerate}[label=(\alph*)]
    \item\label{casea} \(4\) factors of \((N\eta)^{-1/2}\) in case \(\abs{\vl}=0\) and \(\abs{\vj}=1\) for some \(\vj\in J\cup J_\ast\) (implying \(\abs{J\cup J_\ast}\ge 2\)),
    \item\label{caseb} \(3\) factors of \((N\eta)^{-1/2}\) in case \(\abs{\vl}=0\) and \(\abs{J\cup J_\ast}\ge 1\),
    \item\label{casec} \(3\) factors of \((N\eta)^{-1/2}\) in case \(\abs{\vj}=1\) for some \(\vj\in J\cup J_\ast\),
    \item\label{cased} \(2\) factor of \((N\eta)^{-1/2}\) otherwise, 
\end{enumerate}
at the expense of replacing of a multiplicative factor of \(1+\psi_{k_i}^\iso/\sqrt{N\eta}\)
by \(1+(\psi_{2k_i}^\mathrm{iso})^{1/2}/(N\eta)^{1/4}\) for each such improvement.
Indeed, estimating 
\begin{subequations}
    \begin{align} \nonumber
        \sum_{a} \abs{((GA)^n G)_{\vx a}} &\le  \sqrt{N} \sqrt{\sum_{a} \abs{((GA)^n G)_{\vx a}}^2} \le N^{1/2} \sqrt{((GA)^{n}GG^\ast (AG^\ast)^n)_{\vx\vx}} \\\label{iso ward1}
        &\prec \sqrt{\frac{N}{\eta^{n+1}}}\biggl( 1+ \frac{\psi_{2n}^\mathrm{iso}}{\sqrt{N\eta}} \biggr)^{1/2}   \\\nonumber
        \sum_{a} \abs{((GA)^n G)_{\vx a}} \abs{((GA)^m G)_{\vy a}} &\le \sqrt{((GA)^{n}GG^\ast (AG^\ast)^n)_{\vx\vx}}\sqrt{((GA)^{m}GG^\ast (AG^\ast)^m)_{\vx\vx}}\\\label{iso ward2}
        &\prec \frac{1}{\sqrt{\eta^{n+1}}}\biggl( 1+ \frac{\psi_{2n}^\mathrm{iso}}{\sqrt{N\eta}} \biggr)^{1/2}\frac{1}{\sqrt{\eta^{m+1}}}\biggl( 1+ \frac{\psi_{2m}^\mathrm{iso}}{\sqrt{N\eta}} \biggr)^{1/2}    
    \end{align}
\end{subequations}
gains factors of \((N\eta)^{-1/2}\) and  \((N\eta)^{-1}\) respectively, compared to the naive bounds 
\begin{equation}
    \begin{split}
        \sum_{a} \abs{((GA)^n G)_{\vx a}} &\prec \frac{N}{\eta^{n/2}}\biggl( 1+ \frac{\psi_{n}^\mathrm{iso}}{\sqrt{N\eta}} \biggr)^{1/2}   \\
        \sum_{a} \abs{((GA)^n G)_{\vx a}} \abs{((GA)^m G)_{\vy a}} &\prec \frac{N}{\eta^{n/2+m/2}}\biggl( 1+ \frac{\psi_{n}^\mathrm{iso}}{\sqrt{N\eta}} \biggr)^{1/2}\biggl( 1+ \frac{\psi_{m}^\mathrm{iso}}{\sqrt{N\eta}} \biggr)^{1/2},
    \end{split}
\end{equation}
for one and two off-diagonal chains per summation index.  Similar gains are possible for the summation over the $b$-index. 
We call a chain evaluated in \(\vx,a\) or \(\vy,a\) an \(a\)-chain (as in~\eqref{iso ward1}-\eqref{iso ward2}),
and a chain evaluated in \(\vx,b\) or \(\vy,b\) a \(b\)-chain.

We now check that, when performing the $a$  and $b$ summations,
in each of the cases~\cref{casea,caseb,casec,cased} the 
gains~\cref{iso ward1,iso ward2} can be used sufficiently often to obtain the claimed number of $(N\eta)^{-1/2}$ factors. 
Note that even if  there were many \(a\)-chains, a gain is possible from at most two of them. 
\begin{itemize}
    \item[\cref{casea}] Here both the \(\vl\)-factor $[(GA)_{\vx a}(G(AG)^{k-1})_{b\vy}]$ (see~\cref{Xi iso}) and 
    the \(\vj\)-factor $\partial^{\vj}((GA)^kG)_{\vx\vy }$, after performing the derivative,   
    contain  exactly one \(a\)- and one \(b\)-chain each. Hence~\cref{iso ward2} can be used for both summations, 
    and we gain four factors.
    \item[\cref{caseb}] Here the \(\vl\)-factor contains one \(a\)-chain and one \(b\)-chain, while the \(\vj\)-factor contains either an \(a\)- or \(b\)-chain, and thus both~\cref{iso ward1,iso ward2} can be used once for the $a$ and once for the $b$-summation, gaining three factors. 
    \item[\cref{casec}] Due to \(\abs{\vj}=1\), the \(\vj\)-factor contains one \(a\)- and one \(b\)-chain, while the \(\vl\)-factor 
    contains either an \(a\)- or \(b\)-chain, and thus both~\cref{iso ward1,iso ward2} can be used once,
    gaining three factors. 
    \item[\cref{cased}] The \(\vl\)-factor contains either one \(a\)- and one \(b\)-chain, or two \(a\)-chains, or two \(b\)-chains. In the first case we use~\cref{iso ward1} twice, and in the latter two cases we use~\cref{iso ward2} once in order to gain two factors in total.  
\end{itemize}

Now we collect these improvements for~\eqref{Xi iso bound with one Ward}.
If \(\abs{\vl}+\sum (J\cup J_\ast)-\abs{J\cup J_\ast}=0\), then we are in case~\cref{casea} and can gain \(4\) factors. If \(\abs{\vl}+\sum (J\cup J_\ast)-\abs{J\cup J_\ast}=1\), then either \(\abs{\vl}=0\) and we are in case~\cref{caseb}, or \(\abs{\vj}=1\) for all \(\vj\in J\cup J_\ast\) and we are in case~\cref{casec}, yielding three gained factors in both cases.
Finally, if \(\abs{\vl}+\sum (J\cup J_\ast)-\abs{J\cup J_\ast}\ge 2\), then case~\cref{cased} applies 
with a two factor gain. Note that the fewer gains are compensated by the higher power of $1/N$ in 
the prefactor   in~\eqref{Xi iso bound with one Ward}.
Altogether we can conclude that 
\begin{equation}\label{Xi iso bound with Ward}
    \abs{\Xi^\mathrm{iso}_k} \prec \Bigl(\frac{1}{N\eta^{k+1}}\Bigr)^{(1+\abs{J\cup J_\ast})/2} \biggl(1+\frac{\psi_{k-1}^\iso+\psi_k^\iso}{\sqrt{N\eta}}+\frac{(\psi_{k+1}^\iso)^{1/2}+\cdots+(\psi_{2k}^\iso)^{1/2}}{(N\eta)^{1/4}}\biggr)^{\abs{J\cup J_\ast}+1}.
\end{equation}
By plugging~\cref{iso Gaussian bound a priori,Xi iso bound with Ward} into~\cref{eq iso cum exp} we conclude~\cref{eq Psi1iso,eq Psikiso}. 
This proves~\cref{eq Psi1iso,eq Psikiso}. 

For the special $k=2$ case, i.e. for the proof of~\cref{eq Psi2iso} we note that
in the first equality of~\cref{iso der bound} and in the estimate on \(\cE_k^\iso\) there are additional quadratic terms \((\psi_1^\iso)^2\) and \(\psi_1^\iso\psi_1^\av\) but otherwise the proof remains unchanged.\qed

\section{Proof of the reduction inequalities, \cref{lemma doubling} }
\label{sec:redin}

In order to prove~\cref{lemma doubling} we first infer local laws for resolvent chains including some absolute value \(\abs{G}\) from resolvent chains without absolute value. To formulate the precise statement,  for any choices 
of \(g_i(x)\in\set{1/(x-{z_i}),1/\abs{x-z_i}}\) we first generalise~\cref{eq M def} to 
\begin{equation}\label{defM g}
    M(g_1,A_1,g_2,\ldots,A_{k-1},g_{k}) := \sum_{\pi\in\NCP[k]} \pTr_{K(\pi)} (A_1,\ldots,A_{k-1}) \prod_{B\in\pi} \mathrm{sc}_\circ[B],
\end{equation}
where \(\mathrm{sc}_\circ\) is the free cumulant function of \(\mathrm{sc}[i_1,\ldots,i_n]:=\braket{g_{i_1}\cdots g_{i_k}}_\mathrm{sc}\). We note that the bounds~\cref{Msize} and their proofs verbatim also apply to this more generalised \(M\). The following lemma generalises~\cref{G^2 lemma} to absolute values.
\begin{lemma}\label{absG lemma}
    Fix \(\epsilon>0\) and \(\ell,k>0\) and assume that for \(1\le j \le k\) a priori bounds 
    \begin{equation}\label{eq a priori j}
        \Psi_j^\av(\bm z_j,\bm A_j)\prec \psi_j^\av,\qquad \Psi_j^\iso(\bm z_j,\bm A_{j},\vx,\vy)\prec \psi_j^\iso
    \end{equation}
    have been established \emph{\((\epsilon,\ell)\)-uniformly in traceless matrices}. Then \(z_1,\ldots,z_{k+1}\in\C\) with \(\eta=\min_i\abs{\Im z_i}\) and \(G_i\in\set{G(z_i),\abs{G(z_i)}}\) and corresponding \(g_i(x)\in\set{1/(x-z_i),1/\abs{x-z_i}}\) it holds that  
    \begin{equation}\label{eq absG lemma}
        \begin{split}
            \braket{G_1B_1\cdots G_kB_{k}}&=\braket{M(g_1,B_1,\ldots,B_{k-1},g_k)B_k} + \landauOprec*{\frac{\sum_{j=a}^{k}\psi_j^\av \wedge 1}{N\eta^{k-a/2}}}\\
            \Bigl(G_1B_1G_2\cdots B_{k}G_{k+1}\Bigr)_{\vx\vy} &= M(g_1,B_1,\ldots,B_k,g_{k+1})_{\vx\vy} + \landauOprec*{\frac{\sum_{j=a}^{k}\psi_j^\iso\wedge 1}{\sqrt{N}\eta^{k-a/2+1/2}}},
        \end{split}
    \end{equation}
  \((\epsilon,\ell+1)\)-uniformly in vectors $\vx, \vy$  and deterministic matrices \(B_1,\ldots,B_k\), out of which \(a\) are traceless. Furthermore, if all the $B_1,\dots,B_k$ are traceless then \eqref{eq absG lemma} holds \((\epsilon,\ell)\)-uniformly.
\end{lemma}

\begin{proof}
    The proof is analogous to the special case given in~\cref{G^2 lemma}, with the additional step first of representing any \(\abs{G}\) via 
    \begin{equation}\label{absG rep}
        \abs{G(E+\ii\eta)} = \frac{1}{\ii\pi}\int_0^\infty \frac{G(E + \ii(\eta^2+s^2)^{1/2})-G(E - \ii(\eta^2+s^2)^{1/2})}{(\eta^2+s^2)^{1/2}}\dif s
    \end{equation}
    as an integral over resolvents. 
   Here we used the identity 
      \begin{equation}
        \frac{1}{\abs{x-\ii\eta}} = \frac{1}{\ii\pi}\int_0^\infty \biggl(\frac{1}{x-\ii(\eta^2+s^2)^{1/2}}-\frac{1}{x-\ii(\eta^2+s^2)^{1/2}}\biggr)\frac{1}{(\eta^2+s^2)^{1/2}} \dif s.
    \end{equation}
    We note that \(M\) for \(g(x)=\abs{x-E-\ii\eta}^{-1}\) satisfies the analogous identity 
    \begin{equation}\label{M g z equiv}
        M(\ldots,g,\ldots) = \frac{1}{\ii\pi}\int_0^\infty \frac{M(\ldots,E+\ii(\eta^2+s^2)^{1/2},\ldots)-M(\ldots,E-\ii(\eta^2+s^2)^{1/2},\ldots)}{(\eta^2+s^2)^{1/2}}\dif s
    \end{equation}
    by multi-linearity.
        In~\cref{M g z equiv} the lhs.\ is understood in the sense of~\eqref{defM g}, and the rhs.\ in the sense of~\cref{eq M def}. 
    
    It remains to estimate the integral of the error term obtained from using~\cref{absG rep} for each \(\abs{G}\) and replacing the resulting resolvent chains by their deterministic equivalents. 
    From now on we only consider the case $a=k$ in the averaged version (the isotropic one is analogous). Proceeding as in~\cref{G^2 lemma}, the general case $0\le a \le k-1$ is completely analogous and so omitted. The application of~\cref{G^2 lemma} is the only reason why \eqref{eq absG lemma} holds \((\epsilon,\ell+1)\)-uniformly. The proof that now follows for $a=k$ holds  \((\epsilon,\ell)\)-uniformly. For notational simplicity in the following we denote all the deterministic matrices by $A$ and resolvents by $G$ (even if they are evaluated at different spectral parameters). For concreteness we assume that only two $g_i(x)$'s are equal to $|x-z_i|^{-1}$, the rest is $(x-z_i)^{-1}$, i.e. $k_1+k_2+2=k$.  
  Introducing the shorthand notations $z_{i,s}:=E_i+\ii\sqrt{\eta_i^2+t^2}$, $M(z_{1,s},z_{k_1+2,s}):=M(z_{1,s},A,z_2,\dots,z_{k_1+1},A, z_{k_1+2,s},A, z_{k_1+3},\dots, z_k)$, we have
    \begin{equation}
        \begin{split}
            &\big|\braket{|G(E_1+\ii\eta_1)|A(GA)^{k_1}|G(E_{k_1+2}+\ii\eta_{k_1+2})|A(GA)^{k_2}}-\braket{M(g_1,A,\dots,A,g_k)A}\big| \\
            &\lesssim \Bigg|\iint_0^{\infty} \braket{ G(z_{1,s})A(GA)^{k_1} G(z_{k_1+2,s})A(GA)^{k_2}-M(z_{1,s},z_{k_1+2,t})A}\, \frac{\dif s\dif t}{\sqrt{\eta_1^2+s^2}\sqrt{\eta_{k_1+2}^2+t^2}}\Bigg| \\
            &\lesssim\Bigg|\iint_0^{N^{5k}} \braket{ G(z_{1,s})A(GA)^{k_1} G(z_{k_1+2,s})A(GA)^{k_2}-M(z_{1,s},z_{k_1+2,t})A}\, \frac{\dif s\dif t}{\sqrt{\eta_1^2+s^2}\sqrt{\eta_{k_1+2}^2+t^2}}\Bigg| \\
            &\quad+\mathcal{O}\left(N^{-2}\right) \\
            &\prec \frac{\psi_k^\av}{N\eta^{k/2}}\left(\int_0^1\int_0^{N^{5k}}+\int_0^{N^{5k}}\int_0^1\right) \frac{\dif s\dif t}{\sqrt{\eta_1^2+s^2}\sqrt{\eta_{k_1+2}^2+t^2}} \\
            &\quad+\frac{1}{N\eta^k}\int_1^{N^{5k}}\int_1^{N^{5k}} \frac{\dif s\dif t}{\sqrt{\eta_1^2+s^2}\sqrt{\eta_{k_1+2}^2+t^2}} +\mathcal{O}\left(N^{-2}\right) \\
            &\prec \frac{\psi_k^\av\wedge 1}{N\eta^{k/2}}.
        \end{split}
    \end{equation}
    Note that to go from the second to the third line we used the trivial norm bound $\norm{G(E+\ii\eta)}\lesssim \eta^{-1}$ to remove the very large $s$ and $t$ regime (and a similar bound for the deterministic term). Additionally, in the penultimate inequality we used \eqref{eq a priori j} to bound the regime $\eta\le 1$, with $\eta:=\min_i |\Im z_i|$, and the averaged local law~\eqref{loc1} in the regime $\eta\ge 1$. Alternatively, we could have used
  \cite[Theorem 3.4]{MR4372147} in this latter regime. 
\end{proof}

\begin{proof}[Proof of~\Cref{lemma doubling}]
    
    Similarly to Section~\ref{sec master proof}, to make the presentation simpler we do not carry the dependence on the spectral parameters \(z_j\) and traceless matrices \(A_j\) but instead simply write \(G\) and \(A\). 
    
    We first start with the bound in the average case  and
     we distinguish two cases depending on whether $k$ is even or odd. Let $\{\lambda_i\}_{i\in [N]}$ be the eigenvalues of $W$, and let ${\bm u}_i$ be the corresponding eigenvectors. For even $k$, using the shorthand notation $T:=A(GA)^{k/2-1}$, we have
    \begin{equation}
        \label{eq:holde}
        \begin{split}
            &\Psi_{2k}^\mathrm{av} = N\eta^{k} \abs{\braket{(GA)^{2k} -M_{2k}A}} \\
            &\lesssim N\eta+\frac{N\eta^k}{N}\left|\sum_{ijml}\frac{\braket{{\bm u}_i, T{\bm u}_j}\braket{{\bm u}_j, T{\bm u}_m}\braket{{\bm u}_m, T{\bm u}_l}\braket{{\bm u}_l, T{\bm u}_i}}{(\lambda_i-z_1)(\lambda_j-z_{k/2+1})(\lambda_m-z_{k+1})(\lambda_l-z_{(3k)/2+1})}\right| \\
            &\lesssim N\eta+ \frac{N\eta^k}{N}\sum_{ijml}\frac{|\braket{{\bm u}_i, A(GA)^{k/2-1}{\bm u}_j}|^2|\braket{{\bm u}_m, A(GA)^{k/2-1}{\bm u}_l}|^2}{|(\lambda_i-z_1)(\lambda_j-z_{k/2+1})(\lambda_m-z_{k+1})(\lambda_l-z_{(3k)/2+1})|} \\
            &=N\eta+ N^2\eta^k \braket{|G|A(GA)^{k/2-1}|G|A(G^*A)^{k/2-1}}\braket{|G|A(GA)^{k/2-1}|G|A(G^*A)^{k/2-1}} \\
            &\lesssim N\eta+N^2\eta^k\left(\frac{1}{\eta^{k/2-1}}+\frac{\psi_k^\av}{N\eta^{k/2}}\right)^2\le\left(N\eta+\psi_k^\av\right)^2.
        \end{split}
    \end{equation}
   In the last line we used Lemma~\ref{absG lemma} for $a=k$. This concludes the bound for even $k$.

    Similarly, for odd $k$ we have
    \begin{equation}
        \label{eq:hopg}
        \begin{split}
            \Psi_{2k}^\mathrm{av} &= N\eta^{k} \abs{\braket{(GA)^{2k} -M_{2k}A}} \\
            &\lesssim N\eta+  N^2\eta^{k} \braket{|G|A(GA)^{(k+1)/2-1}|G|A(GA)^{(k+1)/2-1}} \\
            &\qquad\quad\times\braket{|G|A(GA)^{(k-1)/2-1}|G|A(GA)^{(k-1)/2-1}} \\
            &\lesssim (N\eta)^2+N\eta(\psi_{k+1}^\mathrm{av}+\psi_{k-1}^\mathrm{av})+\psi_{k+1}^\mathrm{av}\psi_{k-1}^\mathrm{av},
        \end{split}
    \end{equation}
    where to go to the last line we again used Lemma~\ref{absG lemma} for $a=k$. Additionally, to go from the first to the second line of \eqref{eq:hopg} we used (with  the shorthand notation $T:=A(GA)^{(k+1)/2-1}$, $S:=A(GA)^{(k-1)/2-1}$)
    \begin{equation}
        \label{eq:specdec}
        \begin{split}
            &\braket{(GA)^{2k}}\\
            &=\frac{1}{N}\sum_{ijml}\frac{\braket{{\bm u}_i, T{\bm u}_j}\braket{{\bm u}_j, T{\bm u}_m}
            \braket{{\bm u}_m, S{\bm u}_l}\braket{{\bm u}_l, S{\bm u}_i}}{(\lambda_i-z_1)(\lambda_j-z_{(k+1)/2+1})(\lambda_m-w_{k+2})(\lambda_l-w_{(3k+1)/2+1})} \\
            &\lesssim \frac{1}{N}\sum_{ijml}\frac{|\braket{{\bm u}_i, A(GA)^{(k+1)/2-1}{\bm u}_j}|^2|\braket{{\bm u}_m, A(GA)^{(k-1)/2-1}{\bm u}_l}|^2}{|(\lambda_i-z_1)(\lambda_j-z_{(k+1)/2+1})(\lambda_m-w_{k+2})(\lambda_l-w_{(3k+1)/2+1})|} \\
            &= N \braket{|G|A(GA)^{(k+1)/2-1}|G|A(G^*A)^{(k+1)/2-1}}\braket{|G|A(GA)^{(k-1)/2-1}|G|A(G^*A)^{(k-1)/2-1}}.
        \end{split}
    \end{equation}

    We now consider the isotropic case when $k$ is even and $j \ge 1$:

    \begin{equation}
        \label{eq:newrediso}
        \begin{split}
            \Psi_{k+j}^\iso &\lesssim \sqrt{N\eta}+\sqrt{N}\eta^{(k+j+1)/2}\braket{{\bm x}, (GA)^{k+j}G{\bm y}} \\
            &=\sqrt{N\eta}+\sqrt{N}\eta^{(k+j+1)/2}\braket{{\bm x}, (GA)^{k/2} GA(GA)^{j-1}G(AG)^{k/2}{\bm y}} \\
            &\lesssim \sqrt{N\eta}+N\eta^{(k+j+1)/2} \braket{{\bm x}, (GA)^{k/2}|G|(AG^*)^{k/2}{\bm x}}^{1/2}\braket{{\bm y}, (GA)^{k/2}|G|(AG^*)^{k/2}{\bm y}}^{1/2} \\
            &\qquad\qquad\qquad\quad\times \braket{|G|A(GA)^{j-1}|G|(AG^*)^{j-1}A}^{1/2}\\
            &\lesssim \sqrt{N\eta}+N\eta^{(k+j+1)/2} \left(\frac{1}{\eta^{k/2}}+\frac{\psi_k^\iso}{\sqrt{N\eta^{2k+1}}}\right)\left(\frac{1}{\eta^{j-1}}+\frac{\psi_{2j}^\av}{N\eta^j}\right)^{1/2} \\
            &\lesssim \big(N\eta+(N\eta)^{1/2}\psi_k^\iso\big)\big(1+(N\eta)^{-1/2}(\psi_{2j}^\av)^{1/2}\big),
        \end{split}
    \end{equation}
    Additionally, to go from the second to the third line we used that
    \begin{equation}
        \label{eq:goodschwarz}
        \begin{split}
            &\braket{{\bm x}, (GA)^{k/2} GA(GA)^{j-1}G(AG)^{k/2}{\bm y}}  \\
            &\quad = \sum_{ij}\frac{\braket{{\bm x}, (GA)^{k/2} {\bm u}_i}\braket{{\bm u}_i,A(GA)^{j-1}{\bm u}_j}\braket{{\bm u}_j,(AG)^{k/2}{\bm y}}}{(\lambda_i-z_{k/2+1})(\lambda_j-z_{k/2+j+1})} \\
            &\quad \le \left(\sum_{ij} \frac{\braket{|{\bm x}, (GA)^{k/2} {\bm u}_i}|^2|\braket{{\bm u}_j,(AG)^{k/2}{\bm y}}|^2}{|(\lambda_i-z_{k/2+1})(\lambda_j-z_{k/2+j+1})|}\right)^{1/2} \left(\sum_{ij}\frac{|\braket{{\bm u}_i,A(GA)^{j-1}{\bm u}_j}|^2}{|(\lambda_i-z_{k/2+1})(\lambda_j-z_{k/2+j+1})|}\right)^{1/2} \\
            &\quad=N^{1/2}\braket{{\bm x}, (GA)^{k/2}|G|(AG^*)^{k/2}{\bm x}}^{1/2}\braket{{\bm y}, (GA)^{k/2}|G|(AG^*)^{k/2}{\bm y}}^{1/2} \\
            &\quad \quad\quad\times\braket{|G|A(GA)^{j-1}|G|(AG^*)^{j-1}A}^{1/2}.
        \end{split}
    \end{equation}
\end{proof}

\section{Proof of Corollary~\ref{cor:HS}}
The proof of this corollary relies on the Helffer-Sj\"ostrand representation~\cite{MR1345723}, i.e. we express 
each $f_i(W)$ in \(f_1(W)A_1\cdots f_k(W)\) as an integral  of resolvents at different spectral parameters. 
Note that by eigenvalue rigidity (see e.g.~\cite[Theorem 7.6]{MR3068390} or~\cite{MR2871147}) the spectrum of \(W\) is contained in \([-2-\epsilon,2+\epsilon]\), for any small \(\epsilon>0\), with very high probability. In particular this implies that it is enough to consider test functions \(f_i\in H_0^{\lceil k-a/2\rceil}([-3,3])\), i.e. Sobolev functions on $\R$ which are non-zero only on $[-3,3]$.
In fact, this can be always achieved by multiplying the original $f$ with a smooth cut-off function without changing $f(W)$ up to an event of very small probability.

We present the proof only when all the matrices are traceless, i.e. when $a=k$. The proof in the general case is completely analogous and so omitted.

Let \(f\in H_0^{\lceil k/2\rceil}([-3,3])\) then we define its almost analytic extension by
\begin{equation}
    \label{eq:almostanalest}
    f_\mathbf{C}(z)=f_{\mathbf{C},k}(z)=f_{\mathbf{C},k}(x+\ii\eta):=\left[\sum_{j=0}^{\lceil k/2\rceil-1} \frac{(\ii\eta)^j}{j!} f^{(j)}(x)\right]\chi(\eta),
\end{equation}
where \(\chi(\eta)\) is a smooth cut-off equal to one on \([-5, 5]\) and equal to zero on \([-10,10]^c\)
and $f^{(j)}$ denotes the $j$-th derivative. Then we have
\begin{equation}
    \label{eq:HS}
    f(\lambda)=\frac{1}{\pi}\int_\mathbf{C}\frac{\partial_{\overline{z}}f_\mathbf{C}(z)}{\lambda-z}\,\dif^2 z,
\end{equation}
where \(\dif^2 z=\dif x\dif \eta\) denotes the Lebesgue measure on \(\mathbf{C}\equiv\mathbf{R}^2\) with \(z=x+\ii\eta\).

Consider $f_1,\dots,f_k\in H_0^{\lceil k/2\rceil}([-3,3])$, then by~\eqref{eq:HS} we get
\begin{equation}\label{eq:intrep}
    f_1(W)A_1\cdots f_k(W)=\frac{1}{\pi^k}\int_{\C^k} \prod_{i=1}^k\dif^2z_i\left[\prod_{i=1}^k (\partial_{\bar z}(f_i)_\C)(z_i)\right] G(z_1)A_1\cdots A_{k-1}G(z_k),
\end{equation}
where \(G(z_i):=(W-z_i)^{-1}\).

\begin{proof}[Proof of Corollary~\ref{cor:HS}]
    This argument  is very similar to the proof of \cite[Theorem 2.6]{MR4372147}, hence here we only explain 
     the main differences.
  
    Pick any  $\xi>0$ as a tolerance exponent  in the definition of $\landauOprec*{}$.
    Without loss of generality we can assume that \(\max_i \norm{f_i}_{H^{\lceil k/2\rceil}}\lesssim N^{1-\xi}\) 
(otherwise there is nothing to prove). 
We first prove the averaged case in~\eqref{eq:niceintform}, and then we explain the very minor changes required in the isotropic case.
    
    We start with the bound
    \begin{equation}
        \label{eq::bder}
        \int_\mathbf{R} \dif x |\partial_{\overline{z}} f_{\mathbf{C},k}(x+\ii\eta)|\lesssim \eta^{\lceil k/2\rceil-1}\norm{f}_{H^{\lceil k/2\rceil}}
    \end{equation}
    which easily follows from~\eqref{eq:almostanalest}. 
   Set   \(\eta_0:=N^{-1+\xi/2}\);
    first  we prove that the regime \(|\eta_i|\le \eta_0\), for some \(i\in [k]\)  in the integral representation of \(\braket{f_1(W)A_1\dots f_k(W)A_k}\) from~\eqref{eq:intrep} is negligible. Here we only present the proof in the case when \(|\eta_i|\le \eta_0\)
    happens only for a single index $i$; the changes when more than one \(\eta_i\)'s are small are exactly the same as explained above \cite[Eq. (3.21)]{MR4372147}, giving an even smaller bound.
    
    Without loss of generality we assume that \(|\eta_1|\le \eta_0\). In this regime we claim that  (with $z_i=x_i+\ii \eta_i$)
    \begin{equation}
        \label{eq:prelred}
        \begin{split}
            &\Bigg|\int \dif x_1\cdots \dif x_k\int_{\substack{|\eta_i|\ge \eta_0,\\ i\in [2,k]}}\dif \eta_2\cdots\dif \eta_k \int_{-\eta_0}^{\eta_0} \dif \eta_1  \left( \prod_{i=1}^k(\partial_{\bar z}(f_i)_\C)(z_i)\right)\braket{G(z_1)A_1\cdots G(z_k)A_k}\Bigg|\\
            &\qquad\prec \eta_0(N\eta_0)^{k/2-1}\max_i\norm{f_i}_{H^{\lceil k/2\rceil}}.
        \end{split}
    \end{equation}
    
    To prove~\eqref{eq:prelred} we will use Stokes theorem in the following form:
    \begin{equation}
        \label{eq:stokes}
        \int_{-10}^{10}\int_{\widetilde{\eta}}^{10} \partial_{\overline{z}}\psi(x+\ii\eta)h(x+\ii\eta)\, \dif x\dif \eta=\frac{1}{2\ii}\int_{-10}^{10}\psi(x+\ii\widetilde{\eta})h(x+\ii\widetilde{\eta})\, \dif x,
    \end{equation}
    for any \(\widetilde{\eta}\in[0,10]\), and for any \(\psi,h\in H^1(\mathbf{C})\equiv H^1(\mathbf{R}^2)\) such that \(\partial_{\overline{z}}h=0\) on the domain of integration and for \(\psi\) vanishing at the left, right and top boundary of the domain of integration.
   We will use~\eqref{eq:stokes} and the compact support of $(f_i)_\C$ to conclude that
    \begin{equation}
        \label{eq:ineeduse}
        \begin{split}
            &\int_\mathbf{R} \dif x_i \int_{\eta_0}^{10} \dif \eta_i  (\partial_{\overline{z}}(f_i)_\mathbf{C})(z_i) \braket{G(z_1)A_1\dots A_{i-1}G(z_i)A_i\dots G(z_k)A_k}\\
            &\qquad=\frac{1}{2\ii}\int_\mathbf{R} \dif x_i (f_i)_\mathbf{C}(x_i+\ii\eta_0) \braket{G(z_1)A_1\dots A_{i-1}G(x_i+\ii\eta_0)A_i\dots G(z_k)A_k},
        \end{split}
    \end{equation}
    for any fixed \(z_1,\dots, z_{i-1},z_{i+1},\dots,z_k\). Using~\eqref{eq:ineeduse} repeatedly for the \(z_2,\dots,z_k\)-variables,
    we conclude
    \begin{equation}
        \label{eq:intbpbettb}
        \begin{split}
            |\text{lhs.\ of~\eqref{eq:prelred}}|&=\frac{1}{2^{k-1}}\Bigg|\int \prod_{i=1}^k \dif x_i \int_{-\eta_0}^{\eta_0} \dif \eta_1 (\partial_{\overline{z}}(f_1)_\mathbf{C})(x_1+\ii\eta_1)\prod_{i=2}^k (f_i)_\mathbf{C}(x_i+\ii\eta_0) \\
            &\qquad\qquad\qquad\qquad\quad\times \braket{G(z_1)A_1G(x_2+\ii\eta_0)\cdots G(x_k+\ii\eta_0)A_k}|\Bigg|.
         \end{split}
    \end{equation}  
 
 Additionally, we will use the following bound on products of $k$ resolvents which holds uniformly in $|\eta|\ge N^{-10k}$. 
  For this bound we introduce $\rho(z):=\pi^{-1}|\Im m_{\mathrm{sc}}(z)|$, for any $z\in \mathbf{C}\setminus\mathbf{R}$, as the harmonic extension of the semicircle density noting that  $\rho(x+\ii 0)=\rho_\mathrm{sc}(x)$.
    \begin{lemma}
    \label{lem:addbsma}
    For any $k\in \mathbf{N}$, $z_i:=x_i+\ii \eta_i$, with $|x_i|\le 2$ and $|\eta_i|\ge N^{-10k}$, with $i\in [k]$, it holds 
            \begin{equation}
                    \label{eq:newresb}
           |\braket{G(z_1)AG(z_2)\dots  AG(z_k)A}| \prec N^{k/2-1} \prod_{i\in [k]}\frac{1}{\rho(x_i+\ii N^{-2/3})}\left(1+\frac{1}{N|\eta_i|}\right),
        \end{equation}
         \begin{equation}
                    \label{eq:newresbiso}
           |\braket{{\bm x},G(z_1)AG(z_2)\dots  AG(z_k){\bm y}}| \prec N^{(k-1)/2} \prod_{i\in [k]}\frac{1}{\rho(x_i+\ii N^{-2/3})}\left(1+\frac{1}{N|\eta_i|}\right),
        \end{equation}
        uniformly for deterministic 
        traceless matrices $\norm{A}\lesssim 1$, vectors $\norm{\bm x}+\norm{\bm y}\lesssim 1$, and $z_i$ as above.
    \end{lemma}

Armed with all these ingredients, we have  the following chain of inequalities in order to prove~\eqref{eq:prelred}:
   \begin{equation}
        \label{eq:intbpbettb2}
        \begin{split}
            |\text{rhs.\ of~\eqref{eq:prelred}}|&\prec \int \dif x_1 \frac{  | f_1^{(\lceil k/2\rceil)}(x_1) | }{\rho(x_1+\ii N^{-2/3})}\int_{\eta_r\le|\eta_1|\le \eta_0}\eta_1^{\lceil k/2\rceil-1}\left(1+\frac{1}{N|\eta_1|}\right)\dif \eta_1 \left(\prod_{i=2}^k\int\dif x_i \frac{1}{\rho(x_i+\ii N^{-2/3})}\right) 
             \\
            &\quad +\eta_0\norm{f_1}_{H^{\lceil k/2\rceil}} \int_{|\eta_1|\le \eta_r}\eta_1^{\lceil k/2\rceil-2}\eta_0^{-k+1}  \dif \eta_1 \\
            &\lesssim \eta_0 (N\eta_0)^{k/2-1}\left(\int\dif x_1\left|f_1^{(\lceil k/2\rceil)}(x_1)\right|^2\right)^{1/2}
             \left(\int \frac{\dif x_1}{\rho(x_1+\ii N^{-2/3})^2}\right)^{1/2} +\eta_0 \norm{f_1}_{H^{\lceil k/2\rceil}} \\
            &\lesssim \eta_0 (N\eta_0)^{k/2-1}\norm{f_1}_{H^{\lceil k/2\rceil}},
        \end{split}
    \end{equation}
where in the first step we first used
$\norm{f_i}_\infty\lesssim 1$ for $i\in [2,k]$ and  after
  splitting the $\eta_1$ integration,  in the regime $\eta_r\le |\eta_1|\le \eta_0$ we used
     \eqref{eq:newresb} together  with
$$
|\partial_{\overline{z}}(f_1)(x_1+\ii\eta_1)|\lesssim  \eta_1^{\lceil k/2\rceil-1} | f_1^{(\lceil k/2\rceil)}(x_1) |
$$
for any $|x_1|\le 2$, $|\eta_1|\le \eta_0$ 
from \eqref{eq:almostanalest}. 
In the complementary regime  $|\eta_1|<\eta_r$ we used the trivial 
      norm bound  \(|\braket{G(z_1)A_1\cdots G(z_k)A_k}|\le \prod_i \norm{G(z_i)A_i} \le \prod_i |\eta_i|^{-1}\)
       together with \eqref{eq::bder}.  
      In the penultimate inequality of~\eqref{eq:intbpbettb2} we also used that $\int 1/\rho$ is finite 
    due to the square root singularity of $\rho$, and  that $\int 1/\rho^2\lesssim \log N$ thanks to the tiny $N^{2/3}$-regularisation. This concludes the proof of \eqref{eq:prelred}.

    We now estimate the integration regime   in~\eqref{eq:intbpbettb}  where
     \(|\eta_i|\ge \eta_0\) for all \(i\in [k]\). By~\eqref{eq:intrep} and the local law~\eqref{loc1}, we conclude that
    \begin{equation}
        \label{eq:mintegralformhs}
        \begin{split}
            &\braket{f_1(W)A_1\cdots f_k(W)A_k} \\
            &\qquad=\frac{1}{\pi^k}\int_{\R^k}\int_{\eta_0\le |\eta_i|\le 10}\dif^2z_1\cdots \dif^2z_k (\partial_{\bar z}(f_1)_\C)(z_1)\cdots (\partial_{\bar z}(f_k)_\C)(z_k)\braket{M_{[k]}A_k} \\
            &\qquad\quad+\mathcal{O}_\prec\left(\eta_0(N\eta)^{k/2-1}\max_i\norm{f_i}_{H^{\lceil k/2\rceil}}\right).
        \end{split}
    \end{equation}
where we abbreviated $M_{[k]}=M(z_1,A_1,\ldots,z_{k-1},A_{k-1},z_k)$.
    Note that in~\eqref{eq:mintegralformhs} we estimated the error term 
$N^{-1} (\min |\eta_i|)^{-k/2}$ coming from the local law~\eqref{loc1} by
    \begin{equation}
        \label{eq:mintegralformhs2}
        \begin{split}
            &\frac{1}{\pi^k}\int_{\R^k}\int_{\eta_0\le |\eta_i|\le 10}\dif^2{\bm z} \prod_{i=1}^k(\partial_{\bar z}(f_i)_\C)(z_i)\braket{(G(z_1)A_1\dots G(z_k)-M_{[k]})A_k} \\
            &\qquad\qquad\qquad\quad=\mathcal{O}_\prec\left(N^{-1}\max_i\norm{f_i}_{H^{\lceil k/2\rceil}}\right),
        \end{split}
    \end{equation}
    with \(\dif^2{\bm z}:=\dif^2 z_1\dots \dif^2 z_k\). More precisely, in~\eqref{eq:mintegralformhs2} we considered the regime \(\eta_1\le \eta_2\le \cdots \le \eta_k\) (all the other regimes give the same contribution by symmetry) and performed \(k-1\) integration by parts in the \(z_i\)-variables, \(i\in [2,k]\), as in~\eqref{eq:ineeduse}, and then estimated the remaining \(\partial_{\overline{z}}(f_1)_\mathbf{C}(z_1)\) by~\eqref{eq::bder}. The error term $N^{-1}  |\eta_1|^{-k/2}$ from the local law together with the $|\eta_1|^{\lceil k/2\rceil-1}$
    bound from~\eqref{eq::bder} and the integration in $\eta_1$ yields~\eqref{eq:mintegralformhs2}.
    
    Finally, using that by~\eqref{eq:prelred} the regime \(\eta_i\in [\eta_r,\eta_0]\) can be added back to~\eqref{eq:mintegralformhs} at the price of an error \(\eta_0(N\eta_0)^{k/2-1}\max_i\norm{f_i}_{H^{\lceil k/2\rceil}}\)
    we conclude the proof of the averaged case in~\eqref{eq:niceintform} modulo the computation of the leading deterministic term 
which     is done exactly as in \cite[Proof of Theorem 2.6]{MR4372147} and so the details are omitted.
    
    The proof of the isotropic case in~\eqref{eq:niceintform} is very similar. The only differences are the following: (i) to bound the small \(\eta_i\)-regime we have to use \eqref{eq:newresbiso} instead of \eqref{eq:newresb}, which still gives exactly the same bound~\eqref{eq:prelred}; (ii) to estimate the error term coming from the isotropic local law~\eqref{loc2} (used in the regime when \(|\eta_i|\ge \eta_0\) for all \(i\in [k]\)) we have to replace~\eqref{eq:mintegralformhs2} by
    \begin{equation}
        \label{eq:finbas}
        \begin{split}
            &\frac{1}{\pi^k}\int_{\R^k}\int_{\eta_0\le |\eta_i|\le 10}\dif^2{\bm z} \prod_{i=1}^k(\partial_{\bar z}(f_i)_\C)(z_i)\braket{{\bm x},(G(z_1)A_1\dots G(z_k)-M_{[k]}) {\bm y}} \\
            &\qquad\qquad\qquad\quad=\mathcal{O}_\prec\left(N^{-1/2}\max_i\norm{f_i}_{H^{\lceil k/2\rceil}}\right).
        \end{split}
    \end{equation}
    The proof of~\eqref{eq:finbas} is exactly the same as the proof of~\eqref{eq:mintegralformhs2}.
\end{proof}

\appendix
\section{Additional proofs}
\label{sec:addproofs}
\begin{proof}[Proof of \Cref{lemma size M}]
    We first note that the inequality 
    \begin{equation}\label{msc bound}
        \abs{m_\mathrm{sc}[z_1,\ldots,z_j]}\lesssim \frac{1}{\eta^{j-1}} 
    \end{equation}
    is a direct consequence of the integral representation~\cref{msc dd}. The bound~\cref{msc bound} is sharp only when not all \(\Im z_i\) have the same sign. If all signs agree, then the iterated divided difference remains bounded by the smoothness of \(m_\mathrm{sc}\) in the bulk. By M\"obius inversion~\cite[Eq.~(2.3), Lemma 2.16]{MR4372147} we have 
    \begin{equation}\label{m circ}
        \begin{split}
            m_\circ[B] &= \sum_{\pi\in\NCP(B)} (-1)^{\abs{\pi}-1} \biggl(\prod_{S\in K(\pi)}C_{\abs{S}-1}\biggr)\prod_{T\in \pi} m[T]  \\
            &=m[B]+\sum_{\substack{\pi\in\NCP(B)\\\abs{\pi}\ge2}} (-1)^{\abs{\pi}-1} \biggl(\prod_{S\in K(\pi)}C_{\abs{S}-1}\biggr)\prod_{T\in \pi} m[T]\\
            &=m[B] + \sum_{\substack{\pi\in\NCP(B)\\\abs{\pi}\ge2}} \landauO*{\frac{1}{\eta^{\abs{B}-\abs{\pi}}}}\lesssim \frac{1}{\eta^{\abs{B}-1}},
        \end{split}
    \end{equation}
    where \(C_n\) is the \(n\)-th Catalan number. Here we used~\cref{msc bound} in the third and fourth step recalling that 
    \begin{equation}\label{mT}
        m[T]=m\bigl[\set{z_i\given i \in T}\bigr].
    \end{equation}
    We note that~\cref{m circ} is sharp since~\cref{msc bound} is sharp and leading order cancellations are impossible in the ultimate line. 
    
    From the definition~\cref{eq pTr def} it follows that \(\pTr_{K(\pi)}\) is non-zero 
    only when no block of \(K(\pi)\) is a 
    singleton \(\set{i}\) with \(\braket{B_i}=0\), and therefore \(\abs{K(\pi)}\le k - \lceil a/2\rceil\) 
    or equivalently \(\abs{\pi}\ge 1+\lceil a/2\rceil\). Thus~\cref{Msize} 
    follows directly from~\cref{eq M def}.
\end{proof}
\begin{proof}[Proof of~\cref{lemma M rec}]
    We only prove~\cref{M rec 1} as the proof of~\cref{M rec 2} is completely analogous. 
    We recall the alternative definition of \(M\) from~\cite[Eq. (5.12)]{MR4372147}  
    \begin{equation}\label{eq M q def}
        \begin{split}
            \frac{M(z_1,\ldots,z_k)}{m_1\cdots m_k} ={}& \sum_{E\in\NCG[1,k]} \pTr_{K(\pi(E))}(\bm A_{[1,k)}) q_E,\\ 
            q_E:={}&\prod_{e\in E}q_e,\quad q_{ij}:=\frac{m_im_j}{1+m_im_j}, 
            \quad \bm A_S := \bigl( A_i\;\vert\;i\in S),
        \end{split}
    \end{equation}
     where \(\NCG[1,k]\) denotes the set of \emph{non-crossing graphs} on the vertex set \([1,k]=\set{1,\ldots,k}\),
    i.e. graphs  without crossing edges \((ab),(cd)\) with \(a<c<b<d\). The graphs are identified with their edge sets $E$.
    Note that the connected components of any non-crossing graph $E$ form a non-crossing partition of the set  \([1,k]\)
    that we denoted by $\pi(E)$ in~\eqref{eq M q def}.

    For any fixed $j\in [1,k]$, we now partition the set of non-crossing graphs as
    \begin{equation}\label{graphs}
        \begin{split}
            \NCG[1,k] &= 
            \cG_j \sqcup\bigsqcup_{l=1}^{j-1}\bigl(\cG_{lj}^\mathrm{i}\times
            \cG_{lj}^\mathrm{o}\bigr)\sqcup\bigsqcup_{l={j+1}}^{k}\bigl(\cG_{jl}^\mathrm{i}\times\cG_{jl}^\mathrm{o}\bigr),
        \end{split}
    \end{equation}
    according to the idea that each non-crossing graph either
    \begin{enumerate}[label=(\roman*)]
        \item has \(j\) as an isolated vertex, or
        \item has a maximal \(l<j\) with \((lj)\in E\), and the graph can be written as the product of a graph \emph{inside} and a graph \emph{outside} the interval \([l,j]\), or 
        \item has no \(l<j\) with \((lj)\in E\) but there is a maximal \(l>j\) with \((jl)\in E\), 
        and the graph can be written as the product of a graph \emph{inside} and a graph \emph{outside} the interval \([j,l]\).
    \end{enumerate}
  The corresponding formal definitions used  in~\eqref{graphs}   are given 
    \begin{equation}
        \begin{split}
            \cG_j&:=\NCG([1,k]\setminus\set{j}) \\
            \cG_{lj}^\mathrm{i}&:=\NCG[l,j),\qquad \cG_{lj}^\mathrm{o}:= \set{E\in\NCG([1,l]\cup[j,k])\given (lj)\in E}\\
            \cG_{jl}^\mathrm{i}&:=\set{E\in\NCG[j,l]\given (jl)\in E},\qquad \cG_{jl}^\mathrm{o}:= \NCG([1,j)\cup[l,k]).
        \end{split}
    \end{equation}
    We note that for graphs \(E\in\NCG[1,k]\) with an isolated vertex \(j\) whose edge-set is given by the edge-set \(E=E_1\in\cG_j\) of its restriction to \([1,k]\setminus\set{j}\) we have
    \begin{equation}
        \pTr_{K(\pi(E))}(\bm A_{[1,k)}) = \pTr_{K(\pi(E_1))}(\bm A_{[1,j-2]},A_{j-1}A_j,\bm A_{[j+1,k)}).
    \end{equation}
    Similarly for \(E=E_1\cup E_2\) with \(E_1\in\cG_{lj}^\mathrm{i},E_2\in\cG_{lj}^\mathrm{o}\) for some \(l<j\) we have 
    \begin{equation}
        \pTr_{\pi(E)}(\bm A_{[1,k)})=\braket{\pTr_{K(\pi(E_1))} (\bm A_{[l,j-2]})A_{j-1}} \pTr_{K(\pi(E_2))}(\bm A_{[1,l)},I,\bm A_{[j,k)})
    \end{equation}
    since the vertices \(l+1,\ldots,j-1\) are necessarily in distinct connected components than the vertices \(1,\ldots,l-1,j+1,\ldots,k\) due to the non-crossing property. Finally, for \(E=E_1\cup E_2\) with \(E_1\in\cG_{jl}^\mathrm{i},E_2\in\cG_{lj}^\mathrm{o}\) for some \(l>j\) we have 
    \begin{equation}
        \pTr_{\pi(E)}(\bm A_{[1,k)}) = \braket{\pTr_{K(\pi_1)} (\bm A_{[j,l)})}\pTr_{K(\pi(E_2))}(\bm A_{[1,j)},\bm A_{[l,k)})
    \end{equation}
    by the same reasoning. 
    
    Using this decomposition in~\eqref{eq M q def}, we thus obtain 
    \begin{equation}
        \begin{split}\label{M sum}
            &\frac{M(z_1,\ldots,z_k)}{m_1\cdots m_k} \\
            &=  \sum_{E\in\cG_j} q_E \pTr_{K(\pi(E))}(\bm A_{[1,j-2]},A_{j-1}A_j,\bm A_{[j+1,k)})\\
            &\quad + \sum_{l=1}^{j-1}  \sum_{E_1\in\cG_{lj}^\mathrm{i}} q_{E_1}\braket{\pTr_{K(\pi(E_1))} (\bm A_{[l,j-2]})A_{j-1}} \sum_{E_2\in\cG_{lj}^\mathrm{o}} q_{E_2} \pTr_{K(\pi(E_2))}(\bm A_{[1,l)},I,\bm A_{[j,k)})\\
            &\quad + \sum_{l=j+1}^{k}  \sum_{E_1\in\cG_{jl}^\mathrm{i}} q_{E_1}\braket{\pTr_{K(\pi_1)} (\bm A_{[j,l)})} \sum_{E_2\in\cG_{jl}^\mathrm{o}} q_{E_2} \pTr_{K(\pi(E_2))}(\bm A_{[1,j)},\bm A_{[l,k)}).
        \end{split}
    \end{equation}
    
    By~\cref{eq M q def} it follows directly that 
    \begin{equation}\label{E cG id 1}
        \begin{split}
            \sum_{E_1\in\cG_{lj}^\mathrm{i}} q_{E_1} \pTr_{K(\pi(E_1))} (\bm A_{[l,j-2]}) &= \frac{M(z_l,A_l,\ldots,A_{j-2},z_{j-1})}{m_l\cdots m_{j-1}}\\
            \sum_{E_2\in\cG_{jl}^\mathrm{o}} q_{E_2} \pTr_{K(\pi(E_2))}(\bm A_{[1,j)},\bm A_{[l,k)}) &=\frac{M(z_1,\ldots,A_{j-1},z_{l},\ldots,z_{k})}{m_1\cdots m_{j-1}m_l\cdots m_{k}},
        \end{split}
    \end{equation}
    while for \(\cG_{lj}^\mathrm{o}\) and \(\cG_{jl}^\mathrm{i}\) we note that the graphs with or without the edges \((lj)\) or \((jl)\), respectively, give exactly the same tracial expression, and therefore 
    \begin{equation}\label{E cG id 2}
        \begin{split}
            \sum_{E_2\in\cG_{lj}^\mathrm{o}} q_{E_2} \pTr_{K(\pi(E_2))}(\bm A_{[1,l)},I,\bm A_{[j,k)}) &= \frac{q_{lj}}{1+q_{lj}}\frac{M(z_1,\ldots,z_l,I,z_j,\ldots,z_k)}{m_1\cdots m_l m_j\cdots m_k} \\
            \sum_{E_1\in\cG_{jl}^\mathrm{i}} q_{E_1}\braket{\pTr_{K(\pi_1)} (\bm A_{[j,l)})} &=\frac{q_{jl}}{1+q_{jl}}\frac{M(z_j,\ldots,z_l)}{m_1\cdots m_{l}}.
        \end{split}
    \end{equation}
    The claim now follows from using~\cref{E cG id 1,E cG id 2} within~\cref{M sum} and using \(q_{lj}/(1+q_{lj})=m_lm_j\). 
\end{proof}

  \begin{proof}[Proof of Lemma~\ref{lem:addbsma}]
     Let $\epsilon>0$ be arbitrary small and set $J:=N^\epsilon$. For any  $|x|\le 2$, define $z(x,J)=x +\ii\eta(x,J)$ where
     $\eta(x,J)$ is uniquely defined
      implicitly via the equation $N\eta(x,J)\rho( z(x,J))=J$. 
 Note that  $\eta(x,J)\gtrsim N^{-1+\epsilon}$.
      Denote by $\lambda_i$ the eigenvalues of $W$ and by ${\bm u}_i$ the corresponding orthonormal eigenvectors. 
      Additionally, we define the quantiles $\gamma_i$ implicitly by
  \begin{equation}
  \label{eq:defqua}
  \int_{-\infty}^{\gamma_i}\rho_{\mathrm{sc}}(x)\,\dif x=\frac{i}{N}, \qquad i\in [N]
  \end{equation}
  and we recall the {\it rigidity} bound (see e.g. \cite[Theorem 7.6]{MR3068390} or \cite{MR2871147})
  \[
  \big|\lambda_i-\gamma_i\big|\prec \frac{1}{N^{2/3}(N+1-i)^{1/3}}, \qquad 1\le i\le N.
  \]
  Using this eigenvalue rigidity and the  spectral decomposition  of $W$,  it is easy to see
 the following bound on the {\it  overlaps} of the eigenvectors with a test matrix $A$
       \begin{equation}
     \label{eq:newovb}
     \begin{split}
     |\braket{{\bm u}_i,A{\bm u}_j}|^2 & \prec \frac{1}{N\rho(z(\gamma_i,J))\rho(z(\gamma_j,J))}\braket{\Im G(z(\gamma_i,J))A\Im G(z(\gamma_j,J)) A} \\ & \prec \frac{1}{N\rho(z(\gamma_i,J))\rho(z(\gamma_j,J))}
     \end{split}
     \end{equation}
     for any $i,j\in [N]$.  Here we neglected $N^\epsilon$-factors since $\epsilon>0$ is arbitrary 
     small and eventually it can be incorporated in the $\prec$-notation.     
     Note that in the last inequality of \eqref{eq:newovb} we used \eqref{loc1} with $k=a=2$ and
     that the corresponding deterministic term, a linear combination of  $\braket{M(z_i, A_1, z_j)A_2}$ is bounded, see~\eqref{Msize},
      where
     $z_i = z(\gamma_i, J)$ or $z_i = \bar z(\gamma_i, J)$.
     
    Given the overlap bound~\eqref{eq:newovb}, we
     now present the proof of \eqref{eq:newresb}; the proof of \eqref{eq:newresbiso} is completely analogous and so omitted.
     By spectral decomposition for each resolvent together with \eqref{eq:newovb}, using that $\rho(z(x,J))\sim\rho(x+\ii N^{-2/3})$ for any $|x|\le 2$ (modulo $N^\epsilon$-factor), we find that
     \begin{equation}
     \begin{split}
       |\braket{G(z_1)AG(z_2)\dots  AG(z_k)}| &\prec N^{k/2-1} \prod_{j=1}^k\frac{1}{N}\sum_{i=1}^N\frac{1}{|\lambda_i-z_j|
       \rho(\gamma_i+\ii N^{-2/3})}\\
       &\prec N^{k/2-1} \prod_{j=1}^k\frac{1}{\rho(x_j+\ii N^{-2/3})}\left(1+\frac{1}{N|\eta_j|}\right),
       \end{split}
     \end{equation}
     where we used that
              \begin{equation}
     \label{eq:hopb}
     \begin{split}
    \frac{1}{N} \sum_{i=1}^N\frac{1}{|\lambda_i-z_j|\rho(\gamma_i+\ii N^{-2/3})}&\prec  \frac{1}{N}\sum_{|i-i_0|\le N^\delta}\frac{1}{|\lambda_i-z_j|\rho(\gamma_i+\ii N^{-2/3})} \\
    &\quad+ \frac{1}{N}\sum_{|i-i_0|> N^\delta}\frac{1}{|\gamma_i-\gamma_{i_0}|\rho(\gamma_i+\ii N^{-2/3})} \\
    &\prec \frac{1}{\rho(x_j+\ii N^{-2/3})}\left(1+\frac{1}{N|\eta_j|}\right).
     \end{split}
     \end{equation}
     Here $\delta>0$ is an arbitrary small constant (and we neglected $N^\delta$-factors 
     since eventually it can be incorporated in the $\prec$-notation), 
     and $i_0=i_0(j)$ is the index such that $\gamma_{i_0(j)}$ is the closest quantile
      to the fixed $x_j=\Re z_j$. In the first inequality in~\eqref{eq:hopb} we used rigidity to replace $\lambda_i $ and $z_j$
      with the closest quantiles.  In the last step in~\eqref{eq:hopb}  we first used
   that $\rho(\gamma_i+\ii N^{-2/3})$ and $\rho(x_j+\ii N^{-2/3})$ are comparable up to an  $N^\delta$ factor, again by rigidity, 
      and then we used  the trivial bound $1/|\lambda_i-z_j|\le 1/|\eta_j|$ in the first sum and performed the second
      sum using the regular spacing
      of the quantiles.
      \end{proof}

\section{Proof of the multi-resolvent local law in the $d\ge1$ regime}\label{app:largeeta}    
The $d\ge1$ regime is conceptually much simpler than  $d\le 1$ for several reasons. First,
there is no need to keep track of the traceless matrices separately. Second, the trivial 
norm estimate $\| G(z)\|\le 1/ d$ is affordable without much loss. These two facts mean
that long chains of the form $GAGA\ldots G$ can affordably be reduced to much shorter chains
by estimating intermediate $A$ and $G$ factors simply by norm. This trivially takes care of the
reduction problem, the key difficulty in the proof when $d\le 1$; in particular no analogue of  Lemma~\ref{lemma doubling}
is needed. Furthermore, we will not need
to introduce  the quantities $\Psi^{\iso/\av}$ and $\psi^{\iso/\av}$ and gradually improve the estimate on
them; the system of master inequalities reduces to a simple induction on the length $k$ of the 
resolvent chain. 

We will present the proof of the averaged law~\eqref{loc1} for $d\ge1$, the corresponding isotropic 
law~\eqref{loc2} is completely analogous and  will be omitted. The backbone of the argument is 
a very simplified form of Section~\ref{sec master proof}. For notational simplicity, we again
do not carry the precise dependence of the resolvents on the spectral parameters and we
denote every deterministic matrix $A_i$ generically by  $A$. Note that $A$'s are not necessarily traceless.

We prove~\eqref{loc1} by induction on $k$, the initial $k=1$ case will be proven along the way.
We now fix some $k\ge 1$ and in the case $k\ge 2$, we
assume that~\eqref{loc1} has been proven for all resolvent chains of length at most $k-1$.
The starting point of the proof of~\eqref{loc1} for $k$
 is formula~\eqref{eq av un repl0} that we repeat here
\begin{equation}\label{eq av un repl01}
    \begin{split}
        &\braket{(GA)^k} \Bigl(1+\landauOprec*{\frac{1}{Nd^2}}\Bigr)\\
              &\quad = m \braket{A(GA)^{k-1}} + m\sum_{j=1}^{k-1}\braket{(GA)^j G}\braket{(GA)^{k-j}} - m\braket{\un{W(GA)^k}}.        
    \end{split}
\end{equation}
Note that the $1/(N\eta)$ in the error term in the lhs.\ is replaced with $1/(Nd^2)$ since it came from the 
standard single resolvent local law from~\cref{thm sG local law}. 
Notice that all but one chains in the rhs. of~\eqref{eq av un repl01} have less than $k$ resolvents, these can be approximated by their deterministic counterparts using the induction hypothesis of the form
\begin{equation}
    \begin{split}
        \abs*{\braket{A(GA)^{k-1}}- \braket{A M_{k-1}A}} &\prec \frac{1}{Nd^{k}}, \qquad\quad k\ge 2\\
        \abs{\braket{(GA)^jG-M_{j+1}}} &\prec\frac{1}{Nd^{j+2}}, \qquad 1\le j\le k-2 \\ 
        \abs*{\braket{(GA)^{k-j}}-\braket{M_{k-j}A}}&\prec \frac{1}{Nd^{k-j+1}},   \quad j\le k-1.
    \end{split}    
\end{equation}
The $k=1$ case is particularly simple, since the first term in the rhs. of~\eqref{eq av un repl01} is simply $m\braket{A}$
and the sum is absent. In the  $k\ge 2$ case, 
for the remaining \(\braket{(GA)^{k-1}G}\) term we instead use the integral 
representation~\cref{lemma reduc2,lemma reduc2M} in order to also estimate this term using the induction hypothesis as 
\begin{equation}
    \abs{\braket{(GA)^{k-1}G-M_{k}}}\prec \frac{1}{N d^{k+1}}.
\end{equation}

Thus, similarly to the telescopic summation~\eqref{telescope}
and using the deterministic identity~\cref{M rec 1}, we obtain the following analogue of~\eqref{eq av un repl}:
\begin{equation}\label{eq av un repl3}
    \begin{split}
        \braket{(GA)^k - M_kA} = %
         - m\braket{\un{W(GA)^k}}  +\landauOprec*{\wt\cE_k^\av}, \quad \mbox{with}\quad \wt\cE_k^\av:= \frac{1}{Nd^{k+1}},
    \end{split}
\end{equation}
where the error term $\wt\cE_k^\av$ has been appropriately redefined compared with~\eqref{eq av un repl}.

Now we fix any integer $p$  and compute the $2p$-th moment of the lhs. of~\eqref{eq av un repl3}
exactly as in~\eqref{eq av cum exp} with the definition of $\Xi_k^\av$ given in~\eqref{Xi av}.
We follow the calculation from~\eqref{eq av cum exp} through~\eqref{derfinal} but the estimates
are greatly simplified as follows. Instead of~\eqref{av Gaussian bound} we now have
\begin{equation}\label{av Gaussian bound1}
    \begin{split}
      |m|  \frac{\abs{\braket{(GA)^{2k}G}}+\abs{\braket{(GA)^{k}(G^\ast A)^kG^\ast}}}{N^2} \prec  \frac{1}{N^2d^{2k+2}}
        = \Big( \wt\cE_k^\av\Big)^2
    \end{split}
\end{equation}
by a trivial norm bound and $d\ge 1$. Note that we exploited the additional decay $|m|\lesssim 1/d$ 
unlike in~\eqref{av Gaussian bound} where $|m|\lesssim 1$ was used.

Now we turn to the estimate of $\Xi_k^\av$. The naive bounds~\eqref{eq:naiveb} become 
\begin{equation}\label{eq:naiveb1}
         \abs{\partial^{\vl} ((GA)^k)_{ba}} \prec \frac{1}{d^{k+ |\vl|}}, \qquad 
        \abs{\partial^{\vj} \braket{(G^{(\ast)}A)^k}} \prec \frac{1}{Nd^{k+|\vj|}}        
\end{equation}
as long as $\vj\ne 0$,
and they again follow from the trivial norm estimates. 
Using
these bounds in~\eqref{Xi av}, we have
\begin{equation}\label{Xinaive}
\Xi_k^\av \prec N^{-(\abs{\vl}+\sum (J\cup J_\ast)+3)/2} N^2 \frac{1}{d^{k+|\vl|}} \Big( \frac{1}{Nd^{k+1}}\Big)^{\sum (J\cup J_\ast)}
 \le N^{(1-\abs{\vl})/2} \Big(  \wt\cE_k^\av \Big)^{1+\sum (J\cup J_\ast)}.
\end{equation}
If $\abs{\vl}\ge1$, then this naive bound is already sufficient. When $\abs{\vl}=0$, then we perform the $\sum_{ab}$
summation a bit more carefully, similarly to the second line of~\eqref{ex with Ward}:
$$ 
 \sum_{ab} | ((GA)^k)_{ba}| \le N^{3/2} \sqrt{ \braket{ (GA)^{k-1} GG^* (AG^*)^{k-1}}} \prec N^{3/2}d^{-k}.
$$
Note that this bound  gains a factor $1/\sqrt{N}$ compared to the trivial bound 
in~\eqref{Xinaive} since the double sum now contributes only by a factor $N^{3/2}$  instead of $N^2$.
This gain is sufficient to improve~\eqref{Xinaive} to
\begin{equation}\label{Xigood}
\Xi_k^\av \prec \Big(  \wt\cE_k^\av \Big)^{1+\sum (J\cup J_\ast)}.
\end{equation}
Plugging this estimate together  with~\eqref{av Gaussian bound1} into~\eqref{eq av cum exp},
using a Young inequality as we did when going from~\eqref{Y1} to~\eqref{Y2} and recalling
that $p$ was arbitrary, we obtain
$$
 \abs{\braket{(GA)^k - M_kA} }\prec  \wt\cE_k^\av
$$
i.e. we proved~\eqref{loc1} in the $d\ge 1$ regime.

We omit 
the proof of~\eqref{loc2} in the same regime since can be obtained analogously, following
a substantial simplification of the argument in Section~\ref{sec:isso} along the
same lines as the average bound was simplified following Section~\ref{sec:avve}.

\printbibliography%
\end{document}